\documentclass[final,leqno]{siamltex}

\usepackage{graphicx} 
\usepackage{amssymb,bm}
\usepackage{epstopdf}
\usepackage{multirow}
\usepackage{mathtools}
\usepackage{color}
\allowdisplaybreaks
\newtheorem{assumption}{Assumption}
\newtheorem{remark}{Remark}

\newcommand{\R}{\mathbb{R}}
\newcommand{\bx}{\bm{x}}

\title{Stabilized integrating factor Runge-Kutta method and unconditional preservation of maximum \\bound principle}

\author{Jingwei Li\thanks{College of Mathematics and System Science, Xinjiang University, Urumqi 830046, China. Current address: Laboratory of Mathematics and Complex Systems and School of Mathematical Sciences, Beijing Normal University, Beijing 100875, China
({\tt jingwei@bnu.edu.cn}). J. Li's work was partially supported by Excellent Doctor Innovation Program of Xinjiang University grant XJUBSCX-2017006 and National Natural Science Foundation of China grant 61962056.}
        \and Xiao Li\thanks{Department of Applied Mathematics, The Hong Kong Polytechnic University, Hung Hom, Kowloon, Hong Kong ({\tt xiao1li@polyu.edu.hk}). X. Li's work was partially supported by National Natural Science Foundation of China grant 11801024.}
        \and Lili Ju\thanks{Department of Mathematics, University of South Carolina, Columbia, SC 29208, USA ({\tt ju@math.sc.edu}). L. Ju's work was partially supported by US National Science Foundation grant DMS-1818438 and US Department of Energy grant DE-SC0020270.}
        \and Xinlong Feng\thanks{College of Mathematics and System Science, Xinjiang University, Urumqi 830046, China ({\tt fxlmath@xju.edu.cn}). X. Feng's work was partially supported by Research Fund from Key Laboratory of Xinjiang Province grant 2020D04002 and National Natural Science Foundation of China grants U19A2079 and 12071406.}}

\begin{document}

\maketitle

\begin{abstract}
Maximum bound principle (MBP) is an important property for a large class of semilinear parabolic equations, in the sense that
the time-dependent solution of the equation with appropriate initial and boundary conditions and nonlinear operator preserves  for all time a uniform pointwise bound in absolute value. It has been a challenging problem on how to design unconditionally MBP-preserving high-order accurate time-stepping schemes for these equations. In this paper, we combine the integrating factor Runge-Kutta (IFRK) method with the linear stabilization technique
to develop a stabilized IFRK (sIFRK) method, and successfully derive sufficient conditions for the proposed method to
preserve MBP unconditionally in the discrete setting.
We then elaborate some sIFRK schemes with up to the third-order accuracy,
which are proven to be unconditionally MBP-preserving by verifying these conditions.
In addition, it is shown that many classic strong stability-preserving sIFRK schemes
do not satisfy these conditions except the first-order one.
Extensive numerical experiments are also carried out to demonstrate  the performance of the proposed method.
\end{abstract}



\pagestyle{myheadings}
\thispagestyle{plain}
\markboth{J. LI, X. LI, L. JU, AND X. FENG}{sIFRK Method and Unconditional Preservation of MBP}

\section{Introduction}\label{sec1}

Let us consider a class of semilinear parabolic equations of the form 
\begin{eqnarray}
\label{eq1.1}
u_t = \mathcal{L} u + f[u],\qquad t>0,\ \bx\in\Omega,
\end{eqnarray}
where $u=u(t,\bx)$ is the time-dependent quantity of {interest} defined on an open, connected and bounded region $\Omega\subset\R^d$
with Lipschitz boundary $\partial\Omega$,
$\mathcal{L}$ is a linear, local (classic) or nonlocal, elliptic operator, and $f$ represents a nonlinear operator.
For some specific $\mathcal{L}$ and $f$,
the solution to \eqref{eq1.1} under appropriate \textcolor{black}{initial and} boundary conditions satisfies some important properties,
such as  existence of invariant sets and energy dissipation.
The existence of invariant sets also means that
the solution satisfies the maximum bound principle (MBP) \cite{Du4} in the sense that
if the initial data and/or the boundary value are pointwisely bounded by some specific constant in absolute value,
then the absolute value of the solution is also bounded by the same constant everywhere for all time.
A well-known  case is the classic  Allen-Cahn equation \cite{Allen,Evans}
with $\mathcal{L}$ given by the Laplace operator and $f[u]=u-u^3$ in \eqref{eq1.1},
where the constant bounding the solution is $1$.
In addition to the MBP, the Allen-Cahn equation also satisfies the energy dissipation, namely,
the solution decreases some free energy in time.
The energy dissipation is a common property shared by phase-field models,
which are typical cases of the semilinear parabolic equations \eqref{eq1.1}
derived as the gradient flows with respect to some specific free energy functional.
When designing numerical schemes for phase-field models,
the MBP and the energy dissipation are desired to be preserved in the discrete setting
for the equations possessing these two properties.

The MBP becomes an indispensable tool to study physical features of semilinear parabolic equations,
including the aspects of mathematical analysis and numerical simulation.
During the past several decades,
there have been many researches devoted to MBP-preserving numerical methods for equations like \eqref{eq1.1}.
For the spatial discretizations, a partial list includes
the works for finite element method \cite{Burman,Ciarlet2,Xiao1,Xiao2,Yang09},
finite difference method \cite{Chen,Ciarlet1,Varga}, and finite volume method \cite{Peng2,Peng1}.
For the temporal discretizations,
the stabilized linear semi-implicit methods were shown to preserve the MBP
unconditionally for the first-order {schemes} \cite{Shen,TaYa16} but only conditionally for the second-order {methods} \cite{HoLe20}.
Some nonlinear second-order schemes were also presented
to preserve conditionally the MBP for the Allen-Cahn type equations in \cite{HoTaYa17,HoXiJi20}.
The exponential time differencing (ETD) method \cite{Beylkin,Cox,Hochbruck}
was applied to the nonlocal Allen-Cahn equation together with a  linear stabilization technique
and the corresponding first- and second-order ETD schemes were proved to be unconditionally MBP-preserving in \cite{Du3}.
Later, an abstract framework on the MBP-preserving ETD schemes with linear stabilization
was established in \cite{Du4} for a  wide range of semilinear parabolic equations.
The ETD method comes from the variation-of-constants formula
with the nonlinear terms approximated by polynomial interpolations,
followed by exact integration of the resulting integrals involving matrix exponentials.
The stabilized ETD method is efficient and accurate for semilinear parabolic equations with stiff linear and nonlinear terms,
and thus has been successfully applied to various phase-field models recently (see e.g., \cite{Ju1,Ju2,Ju3,ZJZ16}).
{However, as shown in \cite{Du4},
the existing MBP-preserving ETD schemes are only up to second order in time,
while higher-order ETD schemes with stabilization fail to preserve the MBP.}
Therefore,  it is highly desirable to find an alternative choice
to develop higher-order time-stepping schemes which preserve the MBP unconditionally.

The integrating factor (IF) method is another widely-used temporal integration method based on the exponential integrators
and proposed to solve the ordinary differential equations with large Lipschitz constants \cite{Lawson}.
Different from the ETD method,
the IF method is derived by directly applying numerical quadratures to  the integrals in the variation-of-constants formula,
and has been also successfully used for many scientific applications \cite{Hochbruck,Liu,Nie}.
In \cite{Isherwood}, the strong stability-preserving (SSP) integrating factor Runge-Kutta (IFRK) method is proposed for solving \eqref{eq1.1},
where the concept of SSP \cite{Gottlieb} means that
\[
\|u^{n+1}\| \le \|u^n\|,
\]
if the nonlinear operator $f$ satisfies
\begin{equation}
\label{intro_f}
\|u^{n}+\tau f[u^n]\|\leq\|u^n\|, \quad \forall\,\tau\in[0, \tau_{FE}],
\end{equation}
for some $\tau_{FE}>0$ and the linear operator $\mathcal{L}$ satisfies
\begin{equation}
\label{intro_L}
\|e^{\tau\mathcal{L}}\|\leq1, \quad \forall\,\tau\geq0.
\end{equation}
It is observed that the restriction on the time-step size of IFRK methods to be SSP
only comes from the nonlinear term
while the restrictions from both linear and nonlinear parts must be enforced for the standard RK method.
This implies that the IFRK method can be more efficient than the standard RK method,
especially in the case that the linear part of the equation \eqref{eq1.1} is highly stiff.

An initial exploration of high-order MBP-preserving schemes based on the IFRK method was recently made in \cite{Ju4}.
Resorting to  the SSP property  or similarly  the total variation bounded (TVB) property \cite{FeSp05,HuSp11},
several  MBP-preserving IFRK schemes up to the fourth-order accuracy were presented
under the appropriate variants of \eqref{intro_f} and \eqref{intro_L}.
However, all these schemes still need certain constraints on the {time-step} size, which
comes from \eqref{intro_f}.
In this paper, we would like  to completely remove the constraints on the {time-step} size
and develop unconditionally MBP-preserving IFRK schemes.
To this end, one of the key ingredient is the application of the linear stabilization technique.
The stabilization was first introduced in \cite{XuTang} in order
to improve the energy stability of the linear semi-implicit Euler scheme for the phase-field model.
The main idea is to add and subtract a linear term $\kappa u$ in the original equation,
where $\kappa\ge 0$ is a stabilizing constant,
and to make the linear part, combined with the term $\kappa u$, {dominates} the nonlinear part
by choosing the value of $\kappa$ appropriately.
As a result, the stability is improved without sacrificing the linearity of the original semi-implicit scheme.
Apart from the applications of the stabilization technique mentioned in the previous paragraphs,
there {has} been a large amount of literature on the stabilized numerical schemes for phase-field models,
see \cite{FeTaYa13,ShYa10b} and the references therein.

The main contribution of our work in this paper is
to develop  a family of stabilized IFRK (sIFRK) time-stepping schemes for the semilinear parabolic equation \eqref{eq1.1} with unconditional preservation of the MBP. In particular,
we derive sufficient {conditions} for the sIFRK method to preserve the MBP without any constraint on the {time-step} size,
and present some examples of such sIFRK {method} with up to {third-order accuracy} in time.
In addition, we also show {that} the stabilized {versions} of many classic SSP-IFRK schemes developed in \cite{Isherwood}
do not satisfy these conditions except the first-order one.

The rest of the paper is organized as follows.
In Section \ref{sec2}, we briefly review the abstract framework developed  in \cite{Du4} for semilinear parabolic equations,
including the formulation of an equivalent form of \eqref{eq1.1} with linear stabilization  and
the conditions on the linear and nonlinear operators in order to possess the MBP.
In Section \ref{sec3}, we propose the sIFRK method in the general Butcher form and derive the sufficient {conditions} such that
the method can preserve the MBP unconditionally.
Convergence analysis of the sIFRK method is then provided,
as well as energy boundedness.
In addition, we also investigate the SSP-sIFRK method and the corresponding sufficient condition for
unconditional MBP preservation.
Some unconditionally MBP-preserving sIFRK schemes with respectively first-, second-, and third-order temporal accuracies
are then presented and discussed in {detail} in Section \ref{sec4}.
In Section \ref{sec5},  various numerical experiments, including 2D and 3D cases, are performed
to verify the convergence and the unconditional MBP preservation of the proposed  method.
Some concluding remarks are finally given in Section \ref{sec6}.

\section{Overview on maximum bound principle}\label{sec2}

In this section, we give a brief review of the abstract framework established in \cite{Du4}
for analysis of  the maximum bound principle (MBP) of semilinear parabolic equations with the form of \eqref{eq1.1}.
The basic assumptions for the operators will be given along with the main results while all details will be omitted.

For simplicity,  let us consider the semilinear parabolic equation \eqref{eq1.1}
with $\mathcal{L}: C^2(\overline\Omega)\to C(\Omega)$ being the Laplace operator
(or a  second-order elliptic differential operator \cite{Evans00}), subject to the initial condition
\begin{equation}
\label{initial}
u(0,\bx)=u_0(\bx), \quad \bx\in {\overline\Omega}
\end{equation}
and the homogeneous Neumann or the periodic boundary condition (only for a rectangular domain $\Omega=\prod_{i=1}^d(a_i,b_i)$) on $\partial\Omega$.
It is well-known from classic analysis \cite{EnNa00} that
the operator $\mathcal{L}$ generates a  contraction semigroup $\{S_{\mathcal{L}}(t)\}_{t\geq 0}$
with respect to the supremum norm on the subspace of $C(\overline\Omega)$ that satisfies such boundary condition. Next we make the following assumption on the operator $f$.
\begin{assumption}
\label{assump_nonlinear}
The nonlinear operator $f$ acts as a composite function induced by a given one-variable  continuously differentiable function $f_0:\mathbb{R}\to \mathbb{R}$, i.e.,
\begin{eqnarray}
\label{eq2.4}
f[w](\bx)=f_0(w(\bx)),\qquad \forall\,w\in C(\Omega),\ \forall\,\bx\in\Omega,
\end{eqnarray}
and there exists a  constant $\gamma>0$ such that $f_0(\gamma)\leq 0 \leq f_0(-\gamma)$.
\end{assumption}

Then we have
the following result on the MBP for the semilinear parabolic equation \eqref{eq1.1}.

\begin{theorem}
\label{theorem_MBP}
{\rm \cite{Du4}}
Let $T>0$ be a constant.
Under Assumptions \ref{assump_nonlinear}, if $\|u_0\|_{C(\overline{\Omega})}\leq\gamma$,
then the equation \eqref{eq1.1} subject to the homogeneous Neumann or periodic boundary condition
has a unique solution $u\in C([0,T]\times\overline{\Omega})$ and
it satisfies $\|u(t)\|_{C(\overline{\Omega})}\leq\gamma$ for all $t\in[0,T]$.
\end{theorem}

The continuity of a function defined on a set $D\subset\R^d$ is defined as follows \cite{Rudin76}:
\begin{equation*}
\text{$w$ is continuous at $\bx^*\in D$}\iff
\text{$\forall\,\bx_k\to \bx^*$ in $D$ implies $w(\bx_k)\to w(\bx^*)$}.
\end{equation*}
Then under the same analysis framework, the {MBP of Theorem \ref{theorem_MBP}} can be further extended to the case of  finite dimensional operators in space \cite{Du4}, such  as discrete approximations
of $\mathcal{L}$, denoted by $\mathcal{L}^h$,  in which the domain of a function is the set of all spatial grid points (boundary and interior points), denoted by $X$. The corresponding  space-discrete equation of \eqref{eq1.1} with $\mathcal{L}^h$ becomes
an ordinary differential equation (ODE) system taking the same form:
\begin{eqnarray}
\label{eq1.1dis}
u_t = \mathcal{L}^h u + f[u],\qquad t>0,\ \bx\in X^*
\end{eqnarray}
with $u(0,\bx)=u_0(\bx)$ for any $\bx\in X$, where $X^*=X$ for the homogeneous Neumann boundary condition and $X^*=X\cap\overline\Omega_+$ with $\overline\Omega_+=\prod_{i=1}^d(a_i,b_i]$
for the periodic one.
We further assume that the discrete
operator $\mathcal{L}^h$ satisfies the following assumption.
\begin{assumption}
\label{assump_linear}
For any $w\in C(X)$ and $\bx_0\in X^*$, if
\[
w(\bx_0) = \max_{\bx\in{X}} w(\bx),
\]
then $\mathcal{L}^h w(\bx_0)\leq0$.
\end{assumption}

The continuous analogue of Assumption \ref{assump_linear} is obviously satisfied
by  the second-order elliptic differential operator $\mathcal{L}$.
Assumption \ref{assump_linear} guarantees that $\mathcal{L}^h$  generates a contraction semigroup $\{S_{\mathcal{L}^h}(t)\}_{t\geq 0}$  on the subspace of $C(X)$ satisfying the homogeneous Neumann (or periodic) boundary condition.
It is easy to verify that such assumption holds for the discrete approximation of  $\mathcal{L}$
by the classic central difference or mass-lumping finite element method.
Note that in these cases, $\mathcal{L}^h$ can be simply regarded as a square matrix and $S_{\mathcal{L}^h}(t)=e^{t\mathcal{L}^h}$ as  a matrix exponential.
If both Assumptions \ref{assump_nonlinear} and \ref{assump_linear} hold, then  the
 space-discrete problem  of \eqref{eq1.1dis}   has a unique solution
 $u\in C([0,T];C(X))$ satisfying the MBP \cite{Du4}.

\begin{remark}
\label{remark_nonlocal}
As studied in \cite{Du4}, the linear operator $\mathcal{L}$ in  \eqref{eq1.1}
could  be similarly generalized to the nonlocal diffusion operator \cite{Du19} and
the fractional Laplace operator \cite{FeRo16}, and the results of Theorem \ref{theorem_MBP} still hold.
Due to the nonlocality of these two operators, the corresponding boundary conditions are now volume constraints. For the nonlocal diffusion operator, the boundary condition is usually imposed on $\Omega_c$,
a closed and bounded set surrounding $\Omega$ with $\partial\Omega\subset\Omega_c$;
for the fractional Laplace operator, the boundary condition is imposed on $\R^d\setminus\Omega$.
\end{remark}

Let us introduce an artificial stabilizing constant $\kappa>0$.
The space-discrete equation \eqref{eq1.1dis} can be rewritten in the equivalent form
\begin{equation}
\label{eq2.6}
u_t  = \mathcal{L}^h_\kappa u + \mathcal{N}[u],
\end{equation}
where  $\mathcal{L}^h_\kappa =  \mathcal{L}^h - \kappa\mathcal{I}$ and $\mathcal{N}=\kappa\mathcal{I}+f$.
According to \eqref{eq2.4} in Assumption \ref{assump_nonlinear}, we know
\begin{equation*}
\mathcal{N}[w](\bx)=N_0(w(\bx)),\qquad \forall\,w\in C(\Omega),\ \forall\,\bx\in\Omega,
\end{equation*}
where $N_0(\xi)=\kappa\xi+f_0(\xi)$ for $\xi\in \mathbb{R}$.
The stabilizing constant $\kappa$ is required to satisfy
\begin{eqnarray}
\label{assump_kp}
\kappa\geq\max_{|\xi|\leq\gamma}|f_0'(\xi)|,
\end{eqnarray}
which always can be reached since $f_0$ is continuously differentiable.
{Then, the following lemma can be proved.}

\begin{lemma}
\label{L2.1}
{\rm \cite{Du4}}
Under Assumption \ref{assump_nonlinear} and the requirement \eqref{assump_kp}, it holds that
\begin{itemize}
  \item[{\rm (i)}] $|N_0(\xi)|\leq \kappa\gamma$ for any $\xi\in[-\gamma,\gamma]$;
  \item[{\rm (ii)}] $|N_0(\xi_1)-N_0(\xi_2)|\leq 2\kappa|\xi_1-\xi_2|$ for any $\xi_1,\xi_2\in[-\gamma,\gamma]$.
\end{itemize}
\end{lemma}

This lemma plays an important role on the  MBP analysis of time integrations of the equation \eqref{eq1.1}
and its space-discrete system \eqref{eq1.1dis}.
It was shown in \cite{Du3,Du4} that, when applied to the equivalent equation  \eqref{eq2.6} instead of the original  one \eqref{eq1.1dis}, the first- and second-order exponential time differencing (ETD) schemes, ETD1 and ETDRK2 \cite{Cox,ZJZ16},
satisfy the discrete MBP without any restriction on the {time-step} size. The resulting schemes are called
stabilized ETD  schemes for solving the equation \eqref{eq1.1}.
However, such a result cannot be generalized to higher-order (order greater than two) ETD schemes \cite{Du4}.

\section{Unconditionally MBP-preserving stabilized IFRK methods}\label{sec3}

From now on, we suppose all assumptions stated in the previous section hold,
and focus our discussions on time-stepping schemes of the space-discrete system \eqref{eq1.1dis}.

\subsection{Stabilized IFRK schemes and unconditional MBP preservation}\label{sec3.1}

Multiplying both sides of \eqref{eq2.6} by $e^{-t\mathcal{L}^h_\kappa}$ (as exponential integrating factor), we have
\begin{eqnarray*}
e^{-t\mathcal{L}^h_\kappa}(u_t - \mathcal{L}^h_\kappa u) = e^{-t\mathcal{L}^h_\kappa}\mathcal{N}[u],
\end{eqnarray*}
and thus
\begin{eqnarray*}
(e^{-t\mathcal{L}^h_\kappa}u)_t = e^{-t\mathcal{L}^h_\kappa} \mathcal{N}[u].
\end{eqnarray*}
A transformation of variable $w=e^{-t\mathcal{L}^h_\kappa}u$ gives us  the  system
\begin{eqnarray*}
w_t=e^{-t\mathcal{L}^h_\kappa}\mathcal{N}[e^{t\mathcal{L}^h_\kappa}w]=:G(t,w),
\end{eqnarray*}
which is then evolved forward in time from $t_n$ to $t_{n+1}$ using the standard {explicit $s$-stage Runge-Kutta (RK) method \cite{Hairer} ($s$ is a positive integer)}, that is,
\begin{subequations}
\label{rk_w}
\begin{align}
w^{(0)} & = w^n, \\
w^{(i)} & = w^n + \tau \sum_{j=0}^{i-1} a_{ij} G(t_n+c_j\tau,w^{(j)}), \quad  1\leq i\leq s,\label{rk_w2} \\
w^{n+1} & = w^{(s)},
\end{align}
\end{subequations}
where $\tau=t_{n+1}-t_n$ is the uniform time-step size,
\begin{equation}
\label{coef_aij}
a_{ij} \ge 0,\quad  1\leq i\leq s, \ 0\leq j\leq i-1,
\end{equation}
and
\begin{equation}
\label{coef_ci}
c_0=0, \quad  c_i = \sum_{j=0}^{i-1} a_{ij},\ 1\leq i\leq s.
\end{equation}
For the sake of consistency, we also require that $c_s=1$ \cite{Hairer}.
Then, transforming the variable $w$ back to $u$ yields
\begin{subequations}
\label{ifrk_general}
\begin{align}
u^{(0)} & = u^n, \\
u^{(i)} & = e^{c_i\tau\mathcal{L}^h_\kappa} u^n
+ \tau \sum_{j=0}^{i-1} a_{ij} e^{(c_i-c_j)\tau\mathcal{L}^h_\kappa}\mathcal{N}[u^{(j)}], \quad 1\leq i\leq s,\label{ifrk_general2} \\
u^{n+1} & = u^{(s)}.
\end{align}
\end{subequations}
The scheme \eqref{ifrk_general} with the constraints \eqref{coef_aij} and  \eqref{coef_ci} is called  the stabilized integrating factor Runge-Kutta (sIFRK) method for solving the space-discrete system \eqref{eq1.1dis} of the equation \eqref{eq1.1},
in response to the standard IFRK method (i.e., \eqref{ifrk_general} with $\mathcal{L}^h_{\kappa}=\mathcal{L}^h$).

\begin{remark}
Note that the coefficients $\{a_{ij}\}$ and $\{c_i\}$ in  \eqref{ifrk_general}
have slightly different meanings from the usual Butcher table (see, e.g., \cite{Hairer}).
The formula \eqref{rk_w} expresses the RK method in a unified form for each stage, including the last one for $w^{n+1}$, and
the classic Butcher table corresponding to  \eqref{rk_w} takes the following representation:
\begin{equation}
\label{butcher_tab}
\begin{tabular}{c|cccccc}
$c_0$ & $0$ & $0$ & $0$ & $\cdots$ & $0$ & $0$ \\
$c_1$ & $a_{10}$ & $0$ & $0$ & $\cdots$ & $0$ & $0$ \\
$c_2$ & $a_{20}$ & $a_{21}$ & $0$ & $\cdots$ & $0$ & $0$ \\
$\vdots$ & $\vdots$ & $\vdots$ & $\vdots$ & $\ddots$ & $\vdots$ & $\vdots$ \\
$c_{s-2}$ & $a_{s-2,0}$ & $a_{s-2,1}$ & $a_{s-2,2}$ & $\cdots$ & $0$ & $0$ \\
$c_{s-1}$ & $a_{s-1,0}$ & $a_{s-1,1}$ & $a_{s-1,2}$ & $\cdots$ & $a_{s-1,s-2}$ & $0$ \\
\hline
$c_s$ & $a_{s0}$ & $a_{s1}$ & $a_{s2}$ & $\cdots$ & $a_{s,s-2}$ & $a_{s,s-1}$
\end{tabular}
\end{equation}
We still  call \eqref{ifrk_general}  the Butcher form of the sIFRK method.
\end{remark}

Now, we investigate the MBP preservation of the sIFRK method \eqref{ifrk_general}.
To this end, we use the notation $\|\cdot\|$ to represent the vector $\infty$-norm,
and then define the induced matrix $\infty$-norm
as $\|e^{\tau\mathcal{L}^h}\| = \sup\limits_{\|w\|=1}{\|e^{\tau\mathcal{L}^h}w\|}$.
Since $\{e^{\tau\mathcal{L}^h}\}_{\tau\ge0}$ is a contraction semigroup,
which means $\|e^{\tau\mathcal{L}^h}\|\leq 1$ for any $\tau\geq 0$,
and thus,
\begin{equation}
\label{lem_L}
\|e^{\tau\mathcal{L}^h_\kappa}\| \le e^{-\kappa\tau},\qquad\forall\,\tau\geq 0.
\end{equation}

In the following, we present our main result on sufficient conditions for the sIFRK method \eqref{ifrk_general}
in the Butcher form to be unconditionally MBP-preserving.

\begin{theorem}
\label{thm_mbp}
Suppose that  the Butcher table \eqref{butcher_tab} of the sIFRK method satisfies
\begin{itemize}
\item[{\rm (i)}] the property of  nondecreasing abscissas, i.e., $0=c_0\leq c_1\leq c_2\leq \cdots\leq c_s=1$;
\item[{\rm (ii)}] for  $i=1,2,\dots,s$, the function $g_i(x):=e^{-c_ix}+x\sum_{j=0}^{i-1}a_{ij}e^{-(c_i-c_j)x}$
is nonincreasing on $[0,\infty)$.
\end{itemize}
Then, if $\|u^{n}\|\leq\gamma$,
the solution $u^{n+1}$ obtained from \eqref{ifrk_general} always satisfies $\|u^{n+1}\|\leq\gamma$ for any $\tau> 0$.
\end{theorem}

\begin{proof}
By using the condition (i) and the inequality \eqref{lem_L},
it is easy to show  that, for each $i$ in \eqref{ifrk_general2}, we have
\begin{align}\label{eqind}
\|u^{(i)}\|
& \leq \|e^{c_i\tau\mathcal{L}^h_\kappa}\| \|u^n\|
+ \tau\sum_{j=0}^{i-1} a_{ij} \|e^{(c_i-c_j)\tau\mathcal{L}^h_\kappa}\| \|\mathcal{N}[u^{(j)}]\| \nonumber\\
& \leq e^{-c_i\kappa\tau} \|u^n\| + \tau \sum_{j=0}^{i-1} a_{ij} e^{-(c_i-c_j)\kappa\tau} \|\mathcal{N}[u^{(j)}]\|.
\end{align}
Let us assume $\|u^{(j)}\|\leq\gamma$ for all $j\le i-1$. Then,
 we can derive from Lemma \ref{L2.1} and \eqref{eqind} that
\begin{eqnarray}\label{eqres}
\|u^{(i)}\| \leq e^{-c_i\kappa\tau}\gamma + \tau\sum_{j=0}^{i-1}a_{ij}e^{-(c_i-c_j)\kappa\tau}\kappa\gamma
= g_i(\kappa\tau) \gamma.
\end{eqnarray}
Based on the condition (ii), we have $g_i(\kappa\tau)\le g_i(0)=1$, and consequently we have $\|u^{(i)}\|\leq\gamma$ from \eqref{eqres}.
By induction, we obtain $\|u^{(i)}\|\le\gamma$ for $i=1,2,\dots,s$,
and thus, $\|u^{n+1}\|\le\gamma$.
\end{proof}

\begin{remark}
\label{rmk_intstage}
It is observed from the proof that, under the conditions of Theorem \ref{thm_mbp},
if $\|u^n\|\leq\gamma$, then all internal stages of the sIFRK method are also bounded in the norm by $\gamma$, that is, $\|u^{(i)}\|\leq\gamma$ for $1\leq i\leq s$.
Actually, this bound could be sharper, for example,
$\|u^{n+1}\|$ is actually bounded by $g_s(\kappa\tau)\gamma$ instead of $\gamma$.
\end{remark}

Later in Section \ref{sec4}, we will present some examples of unconditionally MBP-preserving sIFRK schemes up to the third-order temporal accuracy  by verifying the conditions  (i) and (ii) in Theorem \ref{thm_mbp}.

\subsection{Convergence  analysis and energy stability}\label{sec3.2}

In the theory of numerical ODEs,
the RK method (\ref{rk_w}) is often called an $s$-stage, $p$th-order method
if the Butcher table (\ref{butcher_tab}) satisfies some appropriate order conditions in the truncation error, see, e.g., \cite{Hairer}.
For simplicity, instead of introducing these order conditions,
we assume that the RK method (\ref{rk_w}) with coefficients \eqref{butcher_tab} possesses the accuracy of order $p$.
Based on this assumption, we now present the error estimates of the sIFRK method (\ref{ifrk_general}).

\begin{theorem}
\label{errthm}
For a fixed $T>0$,
assume that the function $f_0$ in \eqref{eq2.4} is $p$-times continuously differentiable on $[-\gamma,\gamma]$ and
the exact solution $u(t)$ of the space-discrete equation \eqref{eq1.1dis} with the initial data $u_0$ is sufficiently smooth in $[0,T]$.
Let $\{u^n\}$ be the sequence generated by the sIFRK method \eqref{ifrk_general} for \eqref{eq1.1dis}  with $u^0=u_0$.
Under the conditions of Theorem \ref{thm_mbp}, if $\|u_0\|\le\gamma$, then
we have, for any $\tau>0$,
\begin{equation*}
\|u(t_n)-u^n\| \leq C (e^{2\kappa st_n}-1)\tau^p,
\end{equation*}
where the constant $C>0$ is independent of $\tau$.
\end{theorem}

\begin{proof}
Following \cite{Du2},
we introduce the reference functions $U_i$ for $0\le i\le s$, with $U_0=u(t_n)$ and $U_s=u(t_{n+1})$,
determined by
\begin{align}
\label{refer_sol}
U_i & = e^{c_i\tau\mathcal{L}^h_\kappa} u(t_n)
+ \tau \sum\limits_{j=0}^{i-1} a_{ij} e^{(c_i-c_j)\tau\mathcal{L}^h_\kappa} \mathcal{N}[U_j]+R_i, \quad 1\le i\le s,
\end{align}
where $R_i$ is the truncation error satisfying
\begin{align*}
R_0=R_1= \cdots =R_{s-1} = 0, \quad\|R_s\|\leq C_s\tau^{p+1},
\end{align*}
where the constant $C_s>0$ depends on the $C^{p}[0,T]$-norm of $u$, the $C^p[-\gamma,\gamma]$-norm of $f_0$,
$p$, and $\kappa$, but is independent of $\tau$.

Define $e^n=u(t_n)-u^n$ and $e_i=U_i-u^{(i)}$ for $0\le i\le s$, then $e_0=e^n$ and $e_s=e^{n+1}$.
Subtracting (\ref{ifrk_general2}) from (\ref{refer_sol}) yields
\[
e_i = e^{c_i\tau\mathcal{L}^h_\kappa} e^n
+ \tau \sum\limits_{j=0}^{i-1} a_{ij} e^{(c_i-c_j)\tau\mathcal{L}^h_\kappa}(\mathcal{N}[U_j]-\mathcal{N}[u^{(j)}])+R_i, \quad 1\le i\le s.
\]
Since $\|u^{(j)}\|\leq\gamma$ {by Remark \ref{rmk_intstage}},
using Lemma \ref{L2.1}, we can obtain
\[
\|\mathcal{N}[U_j]-\mathcal{N}[u^{(j)}]\| \leq 2\kappa\|U_j-u^{(j)}\|=2\kappa \|e_j\|.
\]
Then for $1\leq i\leq s-1$, we derive
\begin{align*}
\|e_i\| & \leq \|e^{c_i\tau\mathcal{L}^h_\kappa}\|\|e^n
\|+ \tau \sum\limits_{j=0}^{i-1} a_{ij} \|e^{(c_i-c_j)\tau\mathcal{L}^h_\kappa}\|\|\mathcal{N}[U_j]-\mathcal{N}[u^{(j)}]\|\\
& \leq \|e^n\|+2\kappa\tau\sum\limits_{j=0}^{i-1} \|e_j\|, 
\end{align*}
where we have used \eqref{lem_L} and $e^{-\kappa\tau}\leq1$ for any $\tau\geq0$ and $\kappa>0$.
By induction, we can obtain
\[
\|e_i\| \leq (1+2\kappa\tau)^i\|e^n\|, \quad 1\le i\le s-1.
\]
Thus, for $i=s$ we immediately get
\begin{align*}
\|e^{n+1}\| & \le \|e^{\tau\mathcal{L}^h_\kappa} \|\|e^n\|
+ \tau \sum\limits_{j=0}^{s-1} a_{sj} \|e^{(1-c_j)\tau\mathcal{L}^h_\kappa}\|\|\mathcal{N}[U_j]-\mathcal{N}[u^{(j)}]\|+\|R_s\|\\
& \leq \|e^n\| +2\kappa\tau\sum\limits_{j=0}^{s-1} \|e_j\|+C_s\tau^{p+1}\\
& \leq (1+2\kappa\tau)^s\|e^n\|+C_s\tau^{p+1}.
\end{align*}
By induction, we have
\begin{align*}
\|e^{n}\| & \leq(1+2\kappa\tau)^{sn}\|e^0\|+C_s\tau^{p+1}\sum\limits_{i=0}^{n-1}(1+2\kappa\tau)^{si}\\
& \leq (1+2\kappa\tau)^{sn}\|e^0\|+\frac{C_s}{2\kappa s}(e^{2\kappa sn\tau}-1)\tau^{p}.
\end{align*}
By letting $C=\frac{C_s}{2\kappa s}$, we obtain the desired result since $e^0=0$ and $n\tau=t_n$.
\end{proof}

As an application of the convergence result, we next investigate the energy stability.
The semilinear parabolic equation (\ref{eq1.1}) as a
phase-field model can be regarded as the gradient flow driven by the energy
\begin{equation*}
E[u] = -\frac{1}{2} (u,\mathcal{L}u) + (F(u),1),
\end{equation*}
where $F'(u)=-f_0(u)$ and $(\cdot,\cdot)$ denotes the usual $L^2$ inner product in $\Omega$.
The solution to the phase-field model decreases the energy in time until a steady state is reached.
Although we could not prove the energy decay property for numerical solution of \eqref{eq1.1} produced by the sIFRK method,
we can obtain a uniform bound of the energy at any time step.
More precisely, for small enough {time-step} size $\tau$, it holds that
\[
E[u^n] \le E[u_0] + C_0,
\]
where the constant $C_0>0$ is independent of $\tau$.
The proof can be done based on the convergence result from Theorem \ref{errthm}
and  the same process used  in \cite{Du3}, so we omit the details.

\subsection{Strong stability-preserving sIFRK schemes}\label{sec3.3}

As done in \cite{ShOs88}, with  some given $\{\alpha_{ij}\}$ for $1\le i \le s$ and $0\le j\le i-1$  such that
\begin{equation}
\label{ssp_require0}
\alpha_{ij}\geq 0, \quad\sum_{j=0}^{i-1}\alpha_{ij}=1,\ 1\leq i\leq s,
\end{equation}
one can transform the sIFRK method (\ref{ifrk_general}) with \eqref{coef_aij} and \eqref{coef_ci}
in the Butcher form into the following {Shu-Osher} form:
\begin{subequations}
\label{ifrk_ssp}
\begin{align}
u^{(0)} & = u^n, \\
u^{(i)} & = \sum_{j=0}^{i-1} e^{(c_i-c_j)\tau\mathcal{L}^h_\kappa} \big(\alpha_{ij}u^{(j)}+\tau\beta_{ij}\mathcal{N}[u^{(j)}]\big),
\quad  {\rm for }\;\;1\le i\le s, \label{ifrk_ssp2} \\
u^{n+1} & = u^{(s)},
\end{align}
\end{subequations}
where $\beta_{ij}=a_{ij}-\sum_{k=j+1}^{i-1}  \alpha_{ik}a_{kj}$.
Such  formula with $\kappa=0$ (i.e., $\mathcal{L}^h_{\kappa}= \mathcal{L}^h$) has been used in \cite{Isherwood}
to investigate the strong stability-preserving (SSP)  property of numerical schemes for the equation  \eqref{eq1.1},
and later the MBP-preserving property in \cite{Ju4}.

If  all coefficients in  \eqref{ifrk_ssp2} additionally satisfy that
\begin{equation}
\label{ssp_require}
\beta_{ij}\ge 0, \text{ and }
\beta_{ij}=0 \text{ if the corresponding } \alpha_{ij}=0, \quad 0\le j<i\le s,
\end{equation}
then the right-hand side of \eqref{ifrk_ssp2} is clearly a convex combination of a class of integrating factor Euler substeps:
\[
u^{(j)} \mapsto e^{(c_i-c_j)\tau\mathcal{L}^h_\kappa} \Big(u^{(j)}+\tau\frac{\beta_{ij}}{\alpha_{ij}}\mathcal{N}[u^{(j)}]\Big),
\quad 0\leq j\leq i-1.
\]
Thus, we call the time-stepping formula \eqref{ifrk_ssp} with constraints \eqref{ssp_require0} and \eqref{ssp_require} the SSP-sIFRK method, which is the correspondingly
stabilized version  of the SSP-IFRK (i.e., with $\kappa=0$) developed in \cite{Isherwood}.
Obviously, a sIFRK method may not be a SSP-sIFRK method
due to the extra requirement \eqref{ssp_require}.
It is also  worth noting that  the MBP  only holds conditionally for the SSP-IFRK method \cite{Ju4}
in the sense that some restriction on the {time-step} size is needed.

The following result on the sufficient condition for the SSP-sIFRK method to be unconditionally MBP-preserving
can be derived directly based on the Shu-Osher form \eqref{ifrk_ssp}.

\begin{theorem}
\label{thm_mbp_ssp}
Suppose  that the coefficients $\{\alpha_{ij}\}$, $\{\beta_{ij}\}$ satisfy \eqref{ssp_require0} and \eqref{ssp_require} respectively,  \textcolor{black}{and} $\{c_i\}$ satisfies the condition (i) of Theorem \ref{thm_mbp}. In addition, suppose that
for $0\le j<i\le s$,
\begin{equation}
\label{thm_mbp_ssp_cond}
\frac{\beta_{ij}}{\alpha_{ij}}\leq c_i-c_j,
\end{equation}
if  $\alpha_{ij}\not=0$. Then, if $\|u^{n}\|\leq\gamma$,
the solution $u^{n+1}$ obtained from \eqref{ifrk_ssp} always satisfies $\|u^{n+1}\|\leq\gamma$ for any $\tau>0$.
\end{theorem}

\begin{proof}
Assume $\|u^{(j)}\|\le\gamma$ for $j\le i-1$.
Using Lemma \ref{L2.1} and \eqref{lem_L},
we obtain from (\ref{ifrk_ssp2}) that
\begin{align}
\|u^{(i)}\|
& \leq \sum_{j=0}^{i-1} \alpha_{ij} \|e^{(c_i-c_j)\tau\mathcal{L}^h_\kappa}\| \Big\|u^{(j)}+\tau\frac{\beta_{ij}}{\alpha_{ij}}\mathcal{N}[u^{(j)}]\Big\| \nonumber \\
& \leq \sum_{j=0}^{i-1} \alpha_{ij} e^{-(c_i-c_j)\kappa\tau}
\Big(\gamma+\tau\frac{\beta_{ij}}{\alpha_{ij}}\cdot\kappa\gamma\Big) \nonumber \\
& \leq \gamma \sum_{j=0}^{i-1}\frac{\alpha_{ij} }{1+(c_i-c_j)\kappa\tau} \Big(1+\frac{\beta_{ij}}{\alpha_{ij}}\kappa\tau\Big), \label{eqssp}
\end{align}
where we used the fact $e^{-a}\leq\frac{1}{1+a}$ for any $a\geq0$ in the last inequality.
Then, the combination of \eqref{thm_mbp_ssp_cond} and \eqref{eqssp} gives us
\[
\|u^{(i)}\|\le\gamma \sum_{j=0}^{i-1}\alpha_{ij}=\gamma.
\]
By induction, we finally obtain $\|u^{n+1}\|=\|u^{(s)}\|\le\gamma$, which completes the proof.
\end{proof}

\begin{remark}
The sufficient conditions for unconditional MBP preservation of the SSP-sIFRK method given in Theorem \ref{thm_mbp_ssp}
are  easier to check than the ones stated in Theorem \ref{thm_mbp} for testing the case of sIFRK method,
but they  are  not equivalent.
Note that the SSP-sIFRK schemes  \eqref{ifrk_ssp} are mostly obtained by basing the IFRK method
on the optimal canonical Shu-Osher form with non-decreasing abscissas. On the other hand, one also could follow the similar idea as  in \cite{Isherwood} to establish a system of equations and inequalities
with respect to the coefficients $\alpha_{ij}$, $\beta_{ij}$, and $c_i$
based on the conditions \eqref{coef_ci}, \eqref{ssp_require0}, \eqref{ssp_require} and \eqref{thm_mbp_ssp_cond},
and then construct unconditionally MBP-preserving SSP-sIFRK schemes  by solving the optimization problem for the coefficients.
\end{remark}
{\begin{remark}
As an analogue to SSP schemes,
the total variation bounded (TVB) schemes  \cite{FeSp05,HuSp11}
also could  preserve the MBP under certain constraints on the time-step size.  Several conditionally MBP-preserving IFRK schemes found in our recent work \cite{Ju4} are based on the combination of the TVB property and  the IFRK method.
In order to remove these constraints,  one potential way is still to add a linear stabilization term in these schemes as done in this paper.
\end{remark}

\section{Examples of unconditionally MBP-preserving sIFRK method}\label{sec4}

In this section, we present some examples of the sIFRK method which are unconditionally MBP-preserving
by checking the sufficient condition stated in Theorem \ref{thm_mbp} or Theorem \ref{thm_mbp_ssp}.
In addition, we also show that the SSP-sIFRK schemes,
that is, the SSP-IFRK schemes developed in \cite{Isherwood,Ju4} with the proposed stabilization,
fail to hold these conditions except the first-order one.
For simplicity of notations, we denote by sIFRK($s$,$p$) the $s$-stage, $p$th-order sIFRK method
and by the vector $c=[c_0,c_1,c_2,\dots,c_s]^T$ the abscissas.

\subsection{First-order sIFRK scheme}\label{sec4.1}

The sIFRK(1,1) scheme is given by
\[
u^{n+1} = e^{\tau\mathcal{L}^h_\kappa} (u^n + \tau  \mathcal{N}[u^n]).
\]
Here, $c=[0,1]^T$ and  $g_1(x)=e^{-x}+x e^{-x}$ satisfies the conditions (i) and  (ii) in Theorem \ref{thm_mbp},
thus sIFRK(1,1) is unconditionally MBP-preserving.
Note that sIFRK(1,1) is also an SSP-sIFRK method at the same time
since both  \eqref{ssp_require0} and  \eqref{ssp_require} hold.
Moreover, one can find that $g_1(x)<1$ for $x>0$ and $|g_1(x)-1|=\mathcal{O}(x^2)$.

\subsection{Second-order sIFRK schemes}

The family of sIFRK($s$,2) schemes \cite{Suli03} with $c=[0,\frac1s,\frac2s,\dots,1]^T$
(thus the condition (i) in Theorem \ref{thm_mbp} holds) are defined by: $u^{(0)}  = u^n$,
\begin{subequations}
\label{ifrk-s-2}
\begin{align}
u^{(i)}& = e^{\frac{\tau}{s}\mathcal{L}^h_\kappa} \Big(u^{(i-1)}+\frac{\tau}{s}\mathcal{N}[u^{(i-1)}]\Big)\\
& = e^{\frac{i\tau}{s}\mathcal{L}^h_\kappa}u^{n}
+\frac{\tau}{s} \sum_{j=0}^{i-1} e^{\frac{(i-j)\tau}{s}\mathcal{L}^h_\kappa}\mathcal{N}[u^{(j)}],
\quad 1\le i\le s-1, \nonumber \\
u^{n+1} & = e^{\tau\mathcal{L}^h_\kappa} u^n+ \frac{\tau}{s-1} \sum_{j=1}^{s-1} e^{\frac{(s-j)\tau}{s}\mathcal{L}^h_\kappa}\mathcal{N}[u^{(j)}]. \label{ifrk-s-2b}
\end{align}
\end{subequations}
Here, we have
\begin{align*}
g_i(x)& =e^{-\frac{i}{s}x} + \frac{x}{s} \sum_{j=0}^{i-1} e^{-\frac{i-j}{s}x},  \quad  1\le i\le s-1,\\
g_s(x) &= e^{-x} + \frac{x}{s-1} \sum_{j=1}^{s-1} e^{-\frac{s-j}{s}x},
\end{align*}
which can be shown to be nonincreasing on $[0,\infty)$ by checking their derivatives
(i.e., the conditions (ii) in Theorem \ref{thm_mbp} holds). Thus,
all sIFRK($s$,2)  schemes defined by  \eqref{ifrk-s-2} are unconditionally MBP-preserving.

For convenience of use, we list some of sIFRK($s$,2) methods as follows:
\begin{itemize}
\item sIFRK(2,2) with $c=[0, \frac12,1]^T$:
\begin{align*}
u^{(1)} & =e^{\frac{\tau}{2}\mathcal{L}^h_\kappa} \Big(u^n+\frac{\tau}{2}\mathcal{N}[u^n]\Big), \\
u^{n+1} & =e^{\tau\mathcal{L}^h_\kappa}u^n+\tau e^{\frac{\tau}{2}\mathcal{L}^h_\kappa}\mathcal{N}[u^{(1)}].
\end{align*}

\item sIFRK(3,2) with $c=[0, \frac13,\frac23,1]^T$:
\begin{align*}
u^{(1)} & =e^{\frac{\tau}{3}\mathcal{L}^h_\kappa} \Big(u^n+\frac{\tau}{3}\mathcal{N}[u^n]\Big), \\
u^{(2)} & =e^{\frac{\tau}{3}\mathcal{L}^h_\kappa} \Big(u^{(1)}+\frac{\tau}{3}\mathcal{N}[u^{(1)}]\Big), \\
u^{n+1} & =e^{\tau\mathcal{L}^h_\kappa}u^n + \frac{\tau}{2} e^{\frac{2\tau}{3}\mathcal{L}^h_\kappa}\mathcal{N}[u^{(1)}]
+ \frac{\tau}{2} e^{\frac{\tau}{3}\mathcal{L}^h_\kappa} \mathcal{N}[u^{(2)}].
\end{align*}

\item sIFRK(4,2) with $c=[0, \frac14,\frac24,\frac34,1]^T$:
\begin{align*}
u^{(1)} & =e^{\frac{\tau}{4}\mathcal{L}^h_\kappa} \Big(u^n+\frac{\tau}{4}\mathcal{N}[u^n]\Big), \\
u^{(2)} & =e^{\frac{\tau}{4}\mathcal{L}^h_\kappa} \Big(u^{(1)}+\frac{\tau}{4}\mathcal{N}[u^{(1)}]\Big), \\
u^{(3)} & =e^{\frac{\tau}{4}\mathcal{L}^h_\kappa} \Big(u^{(2)}+\frac{\tau}{4}\mathcal{N}[u^{(2)}]\Big), \\
u^{n+1} & =e^{\tau\mathcal{L}^h_\kappa}u^n + \frac{\tau}{3} e^{\frac{3\tau}{4}\mathcal{L}^h_\kappa}\mathcal{N}[u^{(1)}]
+ \frac{\tau}{3} e^{\frac{\tau}{2}\mathcal{L}^h_\kappa} \mathcal{N}[u^{(2)}]
+ \frac{\tau}{3} e^{\frac{\tau}{4}\mathcal{L}^h_\kappa} \mathcal{N}[u^{(3)}].
\end{align*}
\end{itemize}

As pointed out in \cite{Isherwood},
more stages in the methods lead to more accurate numerical results.
We will verify it in the next section by using the second-order methods presented above.

\begin{remark}
\label{rem22}
We note that all of the SSP-sIFRK($s$,2) schemes proposed in \cite{Isherwood} do not satisfy the conditions in
Theorems \ref{thm_mbp} and \ref{thm_mbp_ssp}.
The SSP-sIFRK($s$,2) schemes take the following unified form:  $u^{(0)}=u^n$,
\begin{subequations}
\label{ssps2}
\begin{align}
u^{(i)} & =e^{\frac{\tau}{s-1}\mathcal{L}^h_\kappa}\Big(u^{(i-1)}+\frac{\tau}{s-1} \mathcal{N}[u^{(i-1)}]\Big) \\
& = e^{\frac{i\tau}{s-1}\mathcal{L}^h_\kappa}u^{n}
+\frac{\tau}{s-1} \sum_{j=0}^{i-1} e^{\frac{(i-j)\tau}{{s-1}}\mathcal{L}^h_\kappa}\mathcal{N}[u^{(j)}],  \quad  1\le i\le s-1, \nonumber\\
u^{n+1} & =\frac{1}{s}e^{\tau\mathcal{L}^h_\kappa}u^n+\frac{s-1}{s}\Big(u^{(s-1)}+\frac{\tau}{s-1}\mathcal{N}[u^{(s-1)}]\Big) \\
& = e^{\tau\mathcal{L}^h_\kappa}u^n
+\frac{\tau}{s} \sum_{j=0}^{s-2} e^{\frac{(s-1-j)\tau}{{s-1}}\mathcal{L}^h_\kappa}\mathcal{N}[u^{(j)}]
+\frac{\tau}{s}\mathcal{N}[u^{(s-1)}].  \nonumber
\end{align}
\end{subequations}
Here, $c=[0,\frac{1}{s-1}, \frac{2}{s-1}, \dots ,\frac{s-2}{s-1}, 1,1]^T$.
One can easily see that
\[
\frac{\beta_{s,s-1}}{\alpha_{s,s-1}}=\frac{1}{s-1}>0=c_s-c_{s-1},
\]
which violates \eqref{thm_mbp_ssp_cond} in Theorem \ref{thm_mbp_ssp}.
Moreover, it also does not satisfy the condition (ii) in Theorem \ref{thm_mbp} with
\[
g_s(x)=e^{-x}
+\frac{x}{s} \sum_{j=0}^{s-2} e^{-\frac{s-1-j}{{s-1}}x} + \frac{x}{s}.
\]
\end{remark}

\subsection{Third-order sIFRK schemes}

Unlike the second-order schemes, we do not have a general form for third-order or higher-order schemes.
Below, we present \textcolor{black}{one}  third-order sIFRK \textcolor{black}{scheme}, which \textcolor{black}{satisfies} the conditions (i) and (ii) in Theorem \ref{thm_mbp}
(we simply omit the details of verification), and thus \textcolor{black}{is} unconditionally MBP-preserving.

The \textcolor{black}{scheme} is given by Heun-sIFRK(3,3) with $c=[0,\frac13,\frac23,1]^T$ \cite{Suli03}:
\begin{subequations}
\begin{align*}
u^{(1)} & =e^{\frac{\tau}{3}\mathcal{L}^h_\kappa}u^n + \frac{\tau}{3} e^{\frac{\tau}{3}\mathcal{L}^h_\kappa} \mathcal{N}[u^n],\\
u^{(2)} & =e^{\frac{2\tau}{3}\mathcal{L}^h_\kappa}u^n + \frac{2\tau}{3} e^{\frac{\tau}{3}\mathcal{L}^h_\kappa}\mathcal{N}[u^{(1)}],\\
u^{n+1} & =e^{\tau\mathcal{L}^h_\kappa}u^n + \frac{\tau}{4} e^{\tau\mathcal{L}^h_\kappa}\mathcal{N}[u^n]
+ \frac{3\tau}{4} e^{\frac{\tau}{3}\mathcal{L}^h_\kappa}\mathcal{N}[u^{(2)}]\textcolor{black}{.}
\end{align*}
\end{subequations}

\begin{remark}
The SSP-sIFRK(3,3) scheme proposed in \cite{Isherwood} takes the following form:
\begin{subequations}
\begin{align*}
u^{(1)} &= e^{\frac{2\tau}{3} \mathcal{L}^h_\kappa}\Big(u^{n}+\frac{2}{3} \tau \mathcal{N}[u^{n}]\Big), \\
u^{(2)} &= \frac{2}{3} e^{\frac{2\tau}{3} \mathcal{L}^h_\kappa} u^{n}+\frac{1}{3}\Big(u^{(1)}
+\frac{4}{3} \tau \mathcal{N}[u^{(1)}]\Big), \\
u^{n+1} &= \frac{59}{128} e^{\tau \mathcal{L}^h_\kappa} u^{n}
+\frac{15}{128} e^{\tau \mathcal{L}^h_\kappa}\Big(u^{n}+\frac{4}{3} \tau \mathcal{N}[u^{n}]\Big)
 +\frac{27}{64} e^{\frac{\tau}{3} \mathcal{L}^h_\kappa}\Big(u^{(2)}+\frac{4}{3} \tau \mathcal{N}[u^{(2)}]\Big). \nonumber
\end{align*}
\end{subequations}
Here, $c=[0, \frac23,\frac23,1]^T$. One can easily see
\begin{align*}
\frac{\beta_{2,1}}{\alpha_{2,1}}=\frac43>0=c_2-c_1,
\end{align*}
which violates with \eqref{thm_mbp_ssp_cond} in Theorem \ref{thm_mbp_ssp}.
Moreover, it also does not satisfy the condition (ii) in Theorem \ref{thm_mbp} with
\[
g_2(x)=e^{-\frac23 x} + \frac{2x}{9} e^{-\frac23 x} + \frac{4x}{9}.
\]
\end{remark}

\begin{remark}
We have not found so far in the literature any fourth-order or higher-order  sIFRK scheme in the explicit form
 satisfying the conditions in Theorem \ref{thm_mbp}.
Alternatively, one may further consider the fully implicit IFRK approach,  as done in \cite{Isherwood} for the SSP method, to develop MBP-preserving schemes.
\end{remark}

\section{Numerical experiments}\label{sec5}

In this section, we carry out some numerical experiments to demonstrate the performance of
the sIFRK schemes presented in Section \ref{sec4}.
The spatial discretization is performed by the central difference method for all examples
and the matrix exponentials are implemented by using  the fast Fourier transform (FFT) \cite{Ju2,Ju3}.
First, the convergence rates in time and space are verified by testing a benchmark Allen-Cahn traveling wave problem.
Second, we  verify the unconditional  MBP preservation of the schemes by various evolution examples.
In the end, we present a 3D simulation example to show effectiveness of the proposed method.

\subsection{Convergence tests}

It is well-known that the 2D Allen-Cahn equation in the whole space has a traveling wave solution.
Let us take the domain  $\Omega=(-0.5,0.5)^2$ and  consider the equation
\begin{equation}
\label{eq5.1}
u_t=\Delta u+\frac{1}{\epsilon^2}(u-u^3),\qquad t>0,\ (x,y)\in\Omega
\end{equation}
with the initial data
\begin{equation*}
u_0(x,y)=\frac12\left(1-\tanh\Big(\frac{x}{2\sqrt{2}\epsilon}\Big)\right).
\end{equation*}
The periodic boundary condition is imposed to allow for an approximate traveling wave solution
(for $\epsilon\ll 1$) of the form
\begin{eqnarray*}
u(t,x,y)=\frac12\left(1-\tanh\Big(\frac{x-st}{2\sqrt{2}\epsilon}\Big)\right),
\end{eqnarray*}
where $s=\frac{3}{\sqrt{2}\epsilon}$.
We set $\epsilon=0.015$  and the ending time $T=\frac{\sqrt2\epsilon}{4}$.
In this setting, the stabilizing constant is chosen as $\kappa=\frac{2}{\epsilon^2}$.

Setting $h=1/2048$, we then compute the numerical solutions with various {time-step} sizes by the proposed schemes.
The numerical errors of the solutions at $t=T$
and corresponding convergence rates are given in Tables \ref{tab5.1}--\ref{tab5.3},
where the desired temporal convergence rates (1 for sIFRK(1,1), 2 for sIFRK($s$,2), and 3 for sIFRK($s$,3)) are obviously observed.
{As expected, the error constants are smaller for the schemes with more stages for a fixed order of accuracy.}

\begin{table}[!ht]
\begin{center}
\caption{Errors and convergence rates in time of the traveling wave problem \eqref{eq5.1} using  the first-order sIFRK scheme with
$h=1/2048$ ($\delta=T/128$).}
\label{tab5.1}
\begin{tabular}{|c|c|c|c|c|c|}
\hline
&$\tau$ & $L^2$ Error & Rate & $L^\infty$ Error & Rate\\
\hline\hline\multirow{6}{5em}{sIFRK(1,1)}
&$\delta$&8.4586e-01&--&9.9870e-01&--\\
&$\delta/2$&5.6129e-01&0.59&9.7073e-01&0.04\\
&$\delta/4$&3.5269e-01&0.67&8.1648e-01&0.24\\
&$\delta/8$&2.0037e-01& 0.81&5.3899e-01&0.59\\
&$\delta/16$&1.0573e-01&0.92&3.0023e-01&0.84\\
&$\delta/32$&5.3960e-02&0.97&1.5577e-01&0.95\\
\hline
\end{tabular}
\end{center}
\end{table}

\begin{table}[!ht]
\begin{center}
\caption{Errors and convergence rates in time of the traveling wave problem \eqref{eq5.1} using  the second-order sIFRK   schemes with
$h=1/2048$ ($\delta=T/32$).}
\label{tab5.2}
\begin{tabular}{|c|c|c|c|c|c|c|}
\hline
 & $\tau$ & $L^2$ Error & Rate & $L^\infty$ Error & Rate \\
\hline\hline
\multirow{6}{5em}{sIFRK(2,2)}
&$\delta$&8.4566e-01&--& 9.9917e-01&--\\
&$\delta/2$&4.7032e-01&0.84&9.3871e-01&0.09\\
&$\delta/4$&2.0091e-01&1.22&5.5135e-01&0.76\\
&$\delta/8$&6.3380e-02&1.66&1.8665e-01&1.56\\
&$\delta/16$&1.7488e-02&1.85&5.1933e-02&1.84\\
&$\delta/32$&4.5077e-03&1.95&1.3401e-02&1.95\\
\hline
\multirow{6}{5em}{sIFRK(3,2)}
&$\delta$&7.3155e-01&--&9.9722e-01&--\\
&$\delta/2$&4.0349e-01&0.85&8.8763e-01&0.16\\
&$\delta/4$&1.6074e-01&1.32&4.5403e-01&0.96\\
&$\delta/8$&4.8871e-02&1.71&1.4456e-01&1.65\\
&$\delta/16$&1.3275e-02&1.88&3.9471e-02&1.87\\
&$\delta/32$&3.3790e-03&1.97&1.0052e-02&1.97\\
\hline
\multirow{6}{5em}{sIFRK(4,2)}
&$\delta$&6.7288e-01&--&9.9437e-01&--\\
&$\delta/2$&3.6256e-01&0.89&8.4291e-01&0.23\\
&$\delta/4$&1.3836e-01&1.38&3.9619e-01&1.08\\
&$\delta/8$&4.1285e-02&1.74&1.2239e-01&1.69\\
&$\delta/16$&1.1122e-02&1.89&3.3097e-02&1.88\\
&$\delta/32$&2.8059e-03&1.98&8.3543e-03&1.98\\
\hline
\end{tabular}
\end{center}
\end{table}

\begin{table}[!ht]
\begin{center}
\caption{Errors and convergence rates in time of the traveling wave problem \eqref{eq5.1} using the third-order sIFRK \textcolor{black}{scheme} with
$h=1/2048$ ($\delta=T/16$).}
\label{tab5.3}
\begin{tabular}{|c|c|c|c|c|c|}
\hline
& $\tau$ & $L^2$ Error & Rate & $L^\infty$ Error & Rate \\
\hline\hline
\multirow{6}{5em}{Heun-sIFRK(3,3)}
&$\delta$&8.8453e-01&--&9.9956e-01&--\\
&$\delta/2$ &4.5029e-01&0.97&9.2778e-01&0.10\\
&$\delta/4$ &1.4266e-01&1.65&4.0876e-01&1.18\\
&$\delta/8$ &2.7399e-02&2.38&8.1593e-02&2.32\\
&$\delta/16$ &4.1342e-03&2.72&1.2337e-02&2.72\\
&$\delta/32$ &4.6608e-04&3.14&1.3954e-03&3.14\\
\hline
\end{tabular}
\end{center}
\end{table}

Next, we test the convergence  with respect to the spatial mesh size $h$ using the sIFRK(2,2)  scheme with $\tau=T/2048$.
The numerical errors of the solutions at $t=T$ and corresponding convergence rates are presented in Table \ref{tab5.4},
and it is observed that the spatial convergence is of second order,
which is  consistent with the expectation for the central difference method.

\begin{table}[!ht]
\begin{center}
\caption{Errors and convergence rates in space of the traveling wave problem \eqref{eq5.1} using the sIFRK(2,2) scheme with $\tau=T/2048$.}
\label{tab5.4}
\begin{tabular}{|c|c|c|c|c|}
\hline
$h$ & $L^2$ Error & Rate & $L^\infty$ Error & Rate\\
\hline\hline
$1/32$&3.6280e-01&--&8.2219e-01&--\\
$1/64$&1.2410e-01&1.54&3.3584e-01&1.29\\
$1/128$&3.3108e-02&1.90&9.6036e-02&1.80\\
$1/256$&8.3700e-03&1.98&2.4480e-02&1.97\\
$1/512$&2.0878e-03&2.00&6.1317e-03&1.99\\
$1/1024$&5.2727e-04&1.98&1.5341e-03&1.99\\
\hline
\end{tabular}
\end{center}
\end{table}

\subsection{MBP preservation}

Some examples will be tested to demonstrate  the MBP preservation of  the proposed sIFRK schemes.
The first one focuses on the Allen-Cahn equation with the Flory-Huggins potential consisting of a logarithmic term.
The second one takes simulation of  the classic shrinking bubble example for illustration.
The third one is used to numerically show that the SSP-sIFRK(2,2) \cite{Isherwood} could not hold the MBP unconditionally
as discussed in Remark \ref{rem22}. We still take the domain $\Omega=(-0.5,0.5)^2$ for all experiments in this subsection.

We consider the Allen-Cahn equation
\begin{eqnarray}
\label{eq5.4}
u_t=\epsilon^2\Delta u+f(u)
\end{eqnarray}
subject to periodic boundary condition,
where $\epsilon=0.01$ and $f(u)$ is the negative of the derivative of the Flory-Huggins potential, that is,
\begin{eqnarray*}
f(u)=\frac{\theta}{2}\ln\frac{1-u}{1+u}+\theta_cu,
\end{eqnarray*}
where $\theta=0.8$ and $\theta_c=1.6$.
In this setting, the positive root of the equation $f(\rho)=0$ is $\rho\approx0.9575$,
which is the uniform bound of the exact solution in absolute value,
and the stabilizing constant is chosen as $\kappa=8.02$.

\begin{figure}[!ht]
   \centerline{\includegraphics[width=0.49\textwidth]{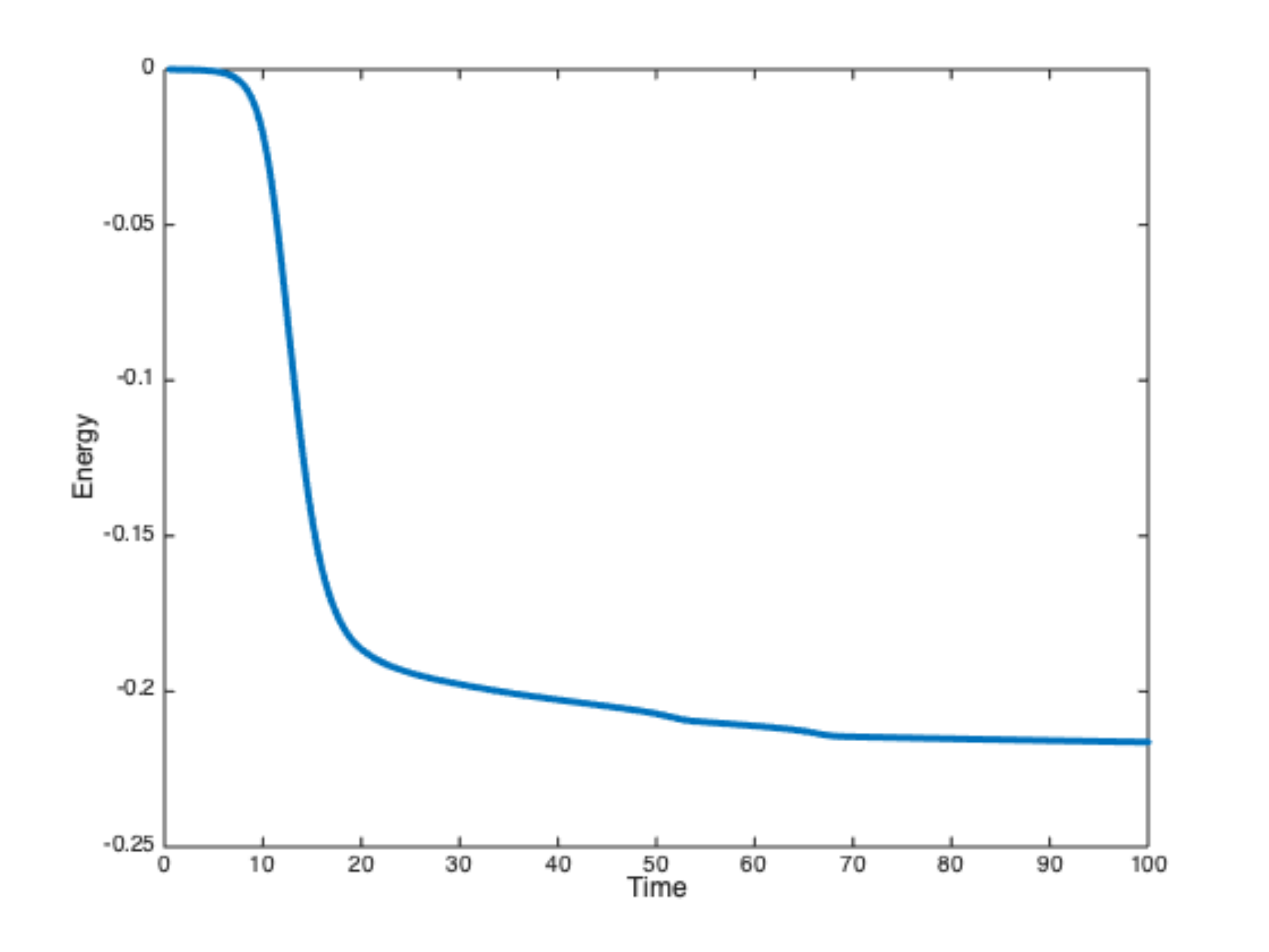}\hspace{-0.5cm}
   \includegraphics[width=0.49\textwidth]{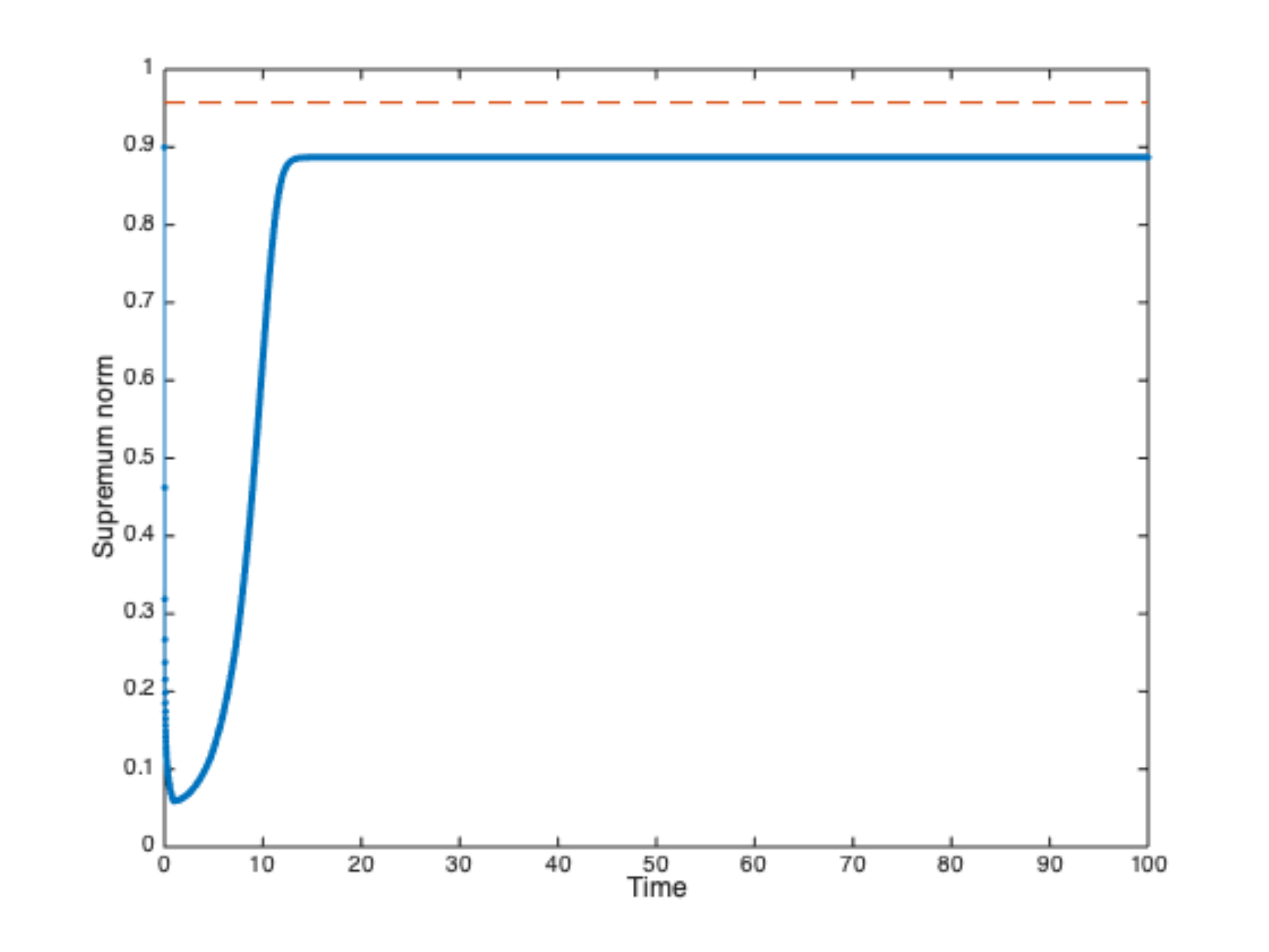}}
   \centerline{\includegraphics[width=0.49\textwidth]{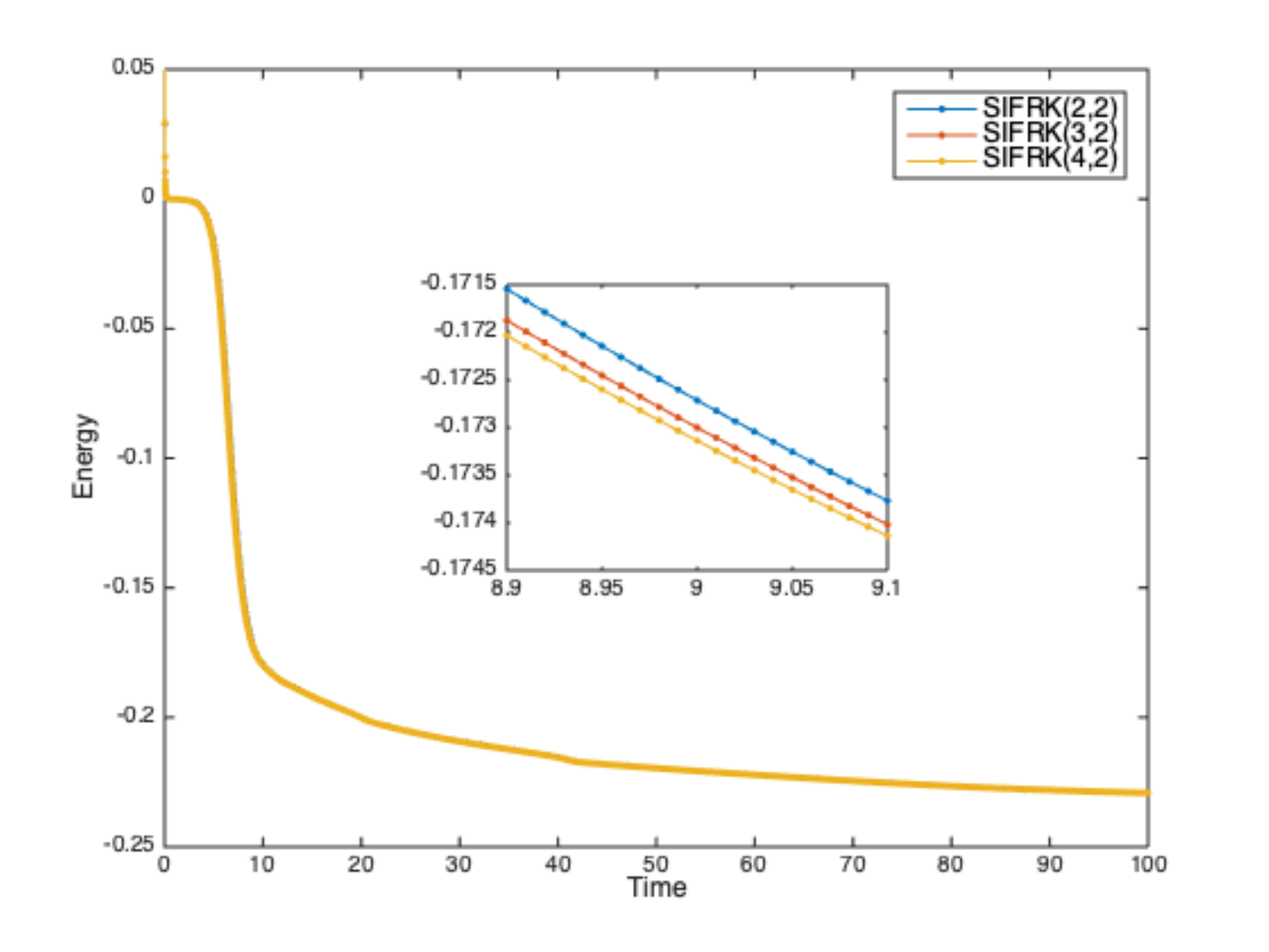}\hspace{-0.5cm}
   \includegraphics[width=0.49\textwidth]{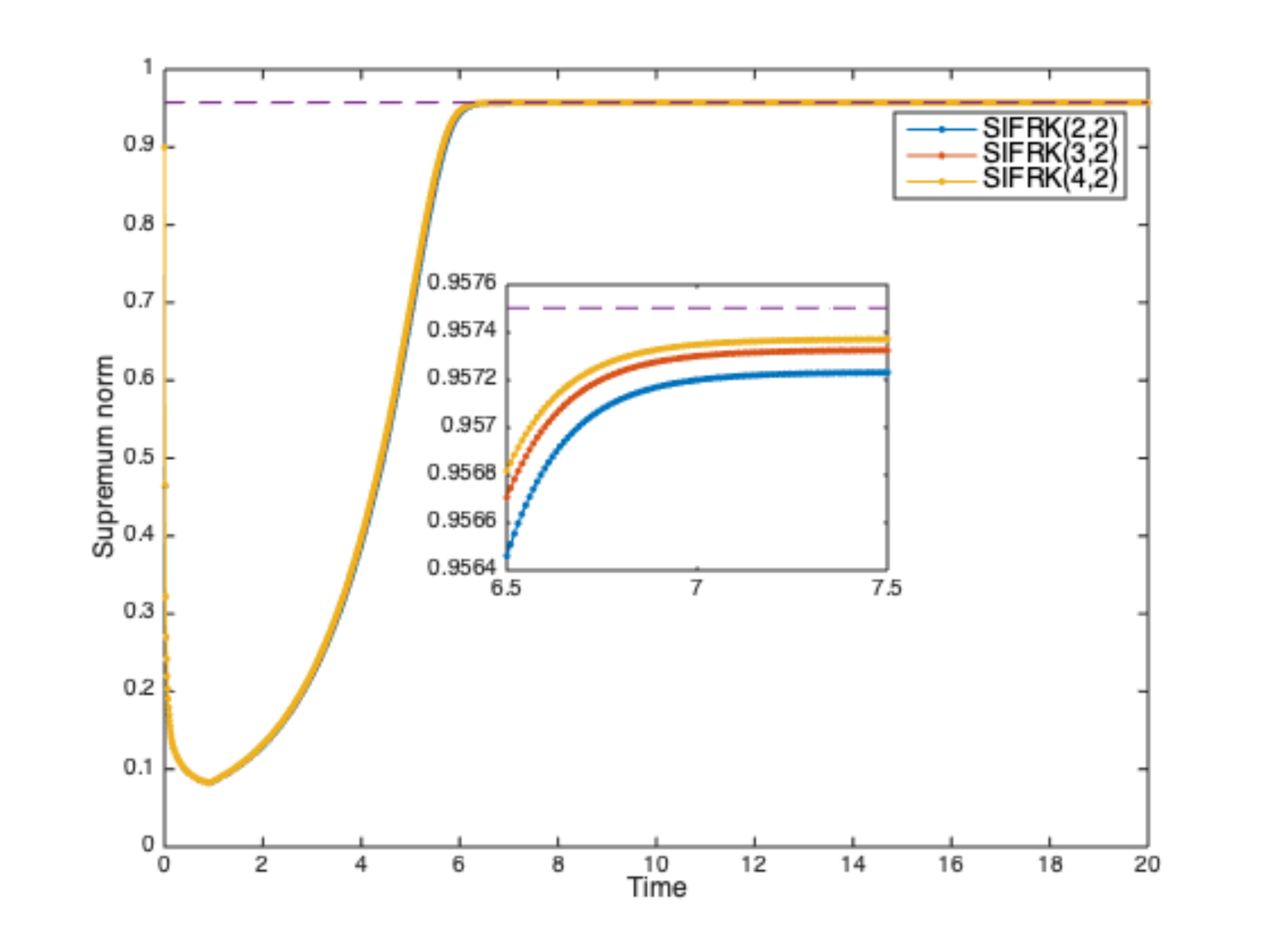}}
   \centerline{\includegraphics[width=0.49\textwidth]{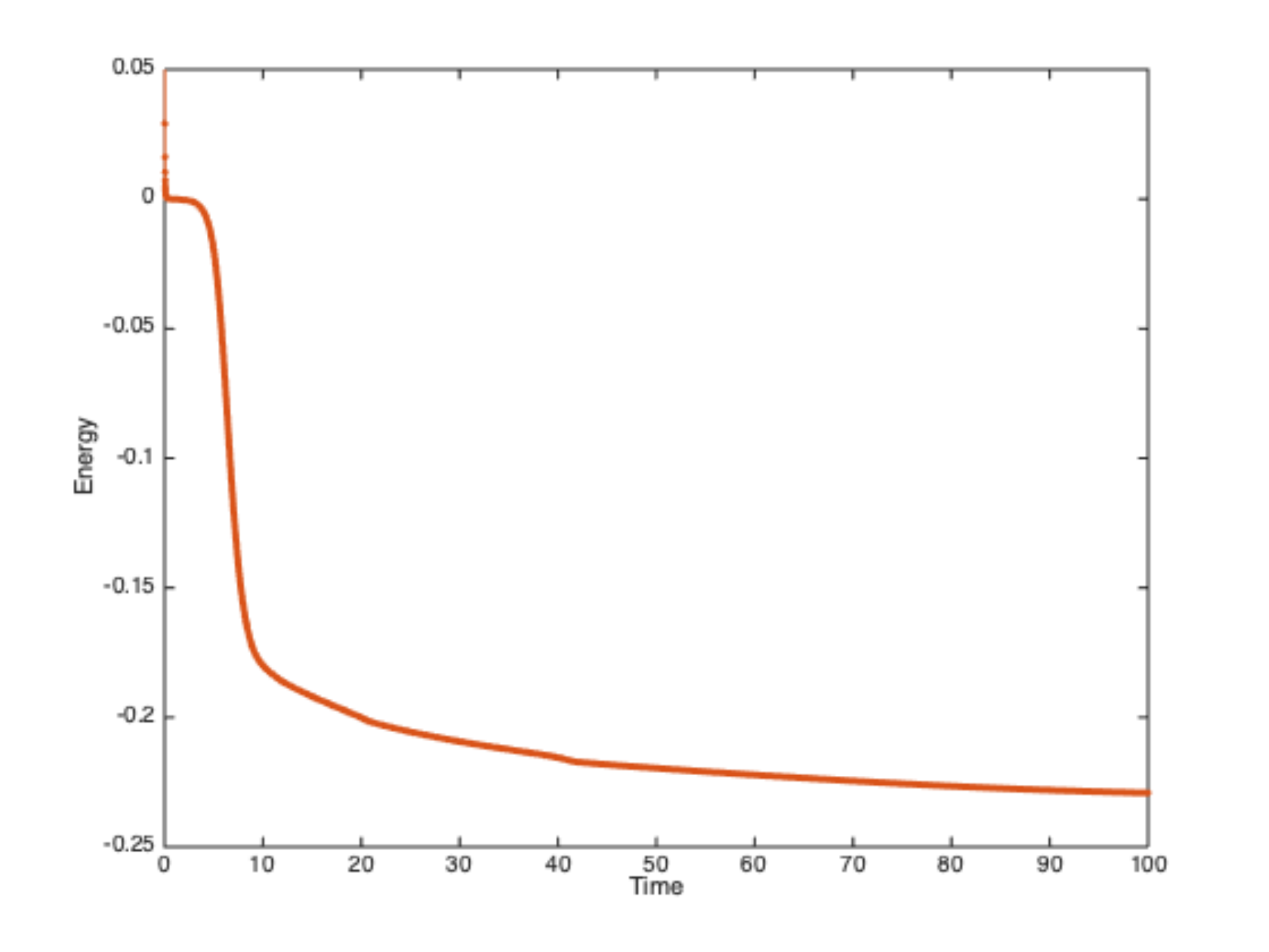}\hspace{-0.5cm}
   \includegraphics[width=0.49\textwidth]{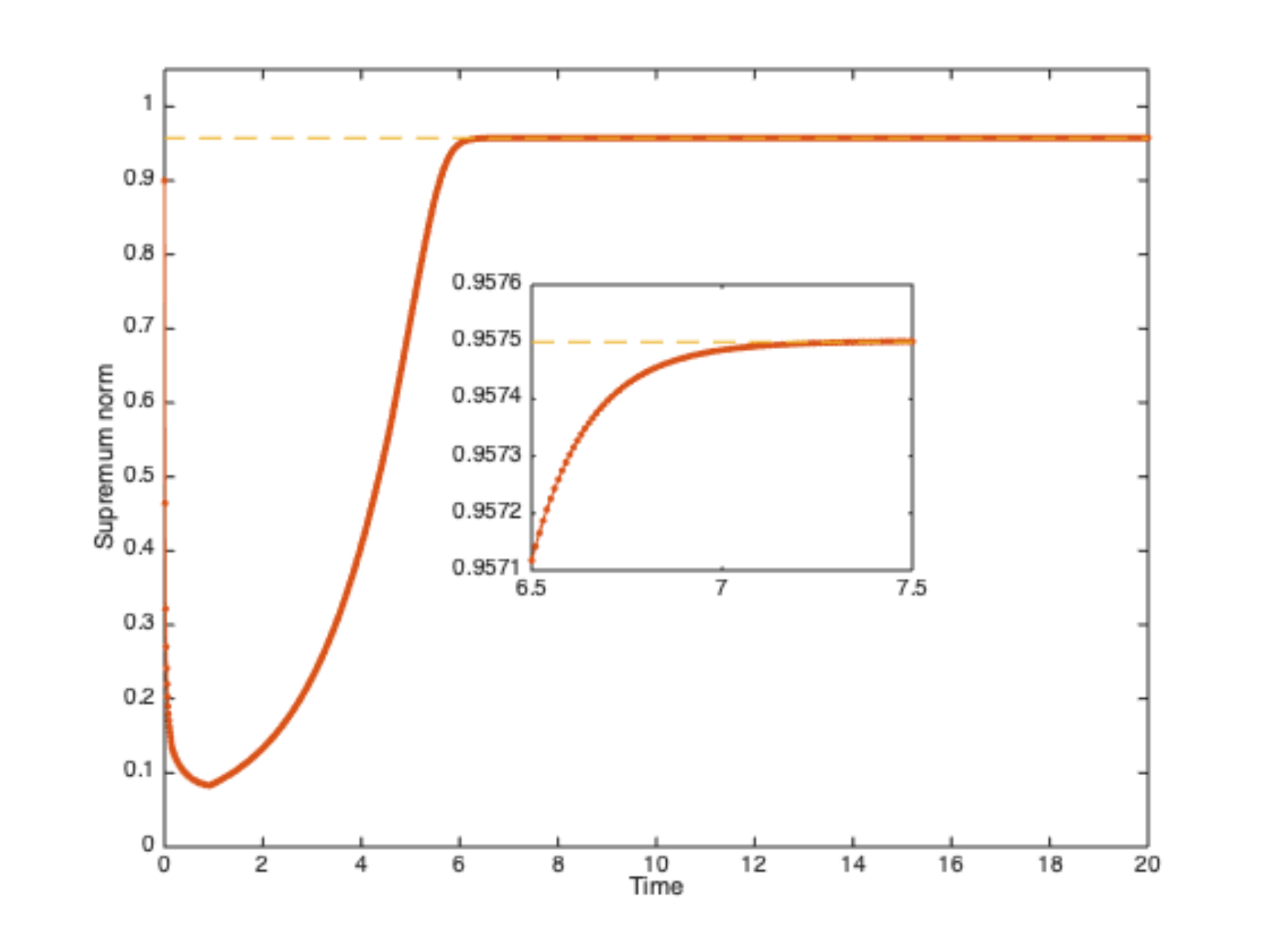}}
    \vspace{-0.2cm}
\caption{Evolutions of energies (left) and supremum norms (right) calculated by using sIFRK(1,1) (top), sIFRK($s$,2) (middle) and \textcolor{black}{Heun-sIFRK(3,3)} (bottom) \textcolor{black}{schemes} with $\tau=0.01$.}
\label{fig5.1}
\end{figure}

We partition the spatial domain by a uniform mesh with the size $h=1/1024$,
and a random data ranging from $-0.9$ to $0.9$ is generated on the mesh as the initial configuration.
We conduct the simulations by using the proposed sIFRK schemes with the {time-step} size $\tau=0.01$.
Fig.~\ref{fig5.1} shows the evolutions of the energies
and the supremum norms of the approximate solutions.
We observe that the energy decreases monotonically
and the MBP is  preserved perfectly  for all of them.
However, there is an obvious large gap between the theoretical bound $0.9575$
and the supremum norm of the steady state obtained by sIFRK(1,1).
The reason is what we have mentioned in Remark \ref{rmk_intstage}, that is,
$g_1(\kappa\tau)<1$ and the difference is of order $\mathcal{O}((\kappa\tau)^2)$ as discussed in Section \ref{sec4.1}.
This implies that the first-order sIFRK method is practically not accurate
although stable when the time-step size $\tau$ is not small enough.

Now, we consider a classic example for simulating a shrinking bubble
driven by the Allen-Cahn equation \eqref{eq5.4} with  $\epsilon=0.01$
and
\begin{equation}
\label{eq5.6}
f(u)=u-u^3
\end{equation}
subject to the homogeneous Neumann boundary condition.
The initial bubble is given by
\begin{equation*}
u_0(x,y)=\left\{
\begin{array}{rl}
1, & \text{if }x^2+y^2\leq 0.25^2,\\
-1, & \text{otherwise},
\end{array}
\right.
\end{equation*}
and  illustrated in \figurename~\ref{fig5.7}.

\begin{figure}[!ht]
\centerline{
   \includegraphics[width=0.49\textwidth]{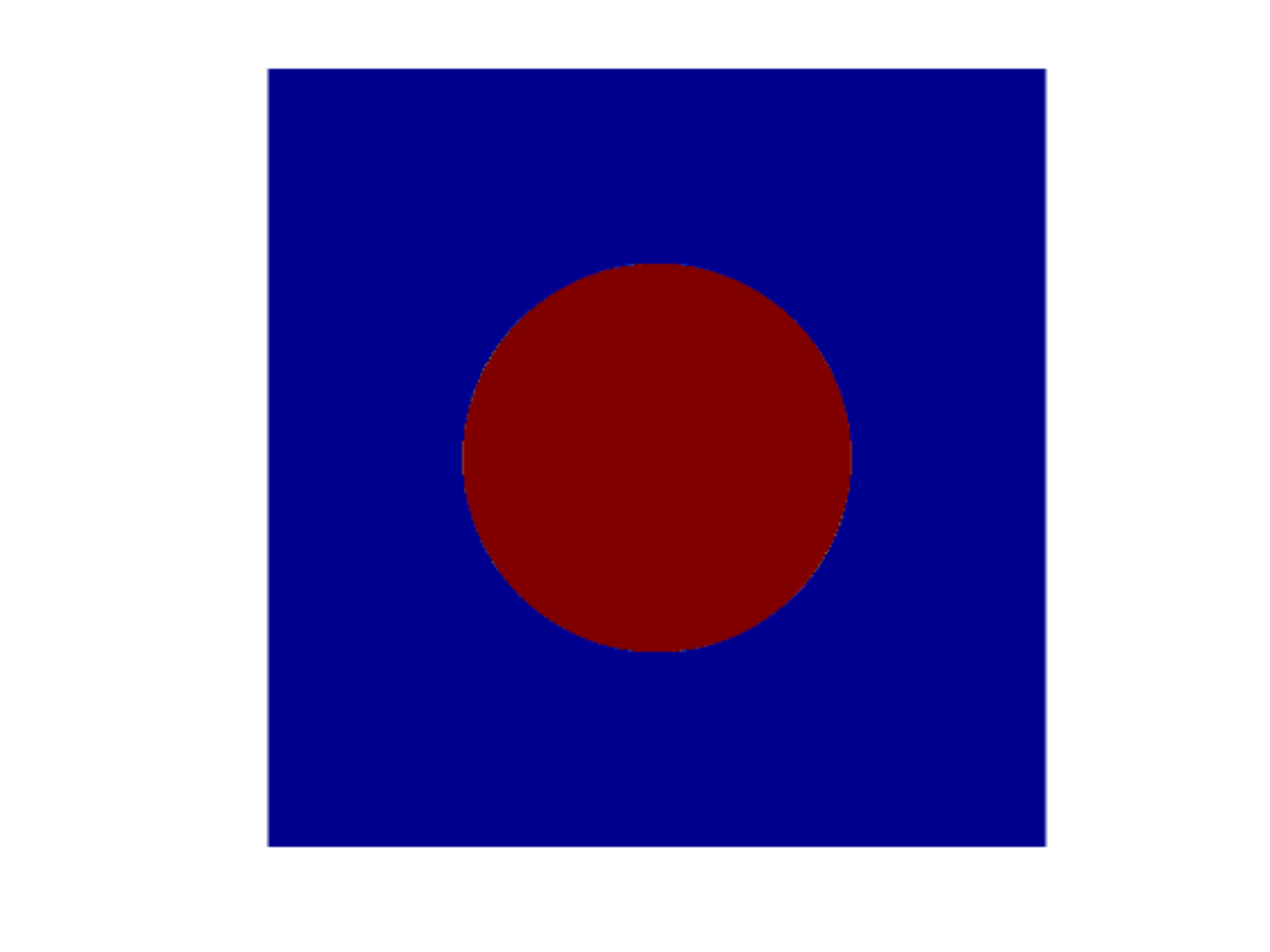}\hspace{-0.5cm}
   \includegraphics[width=0.49\textwidth]{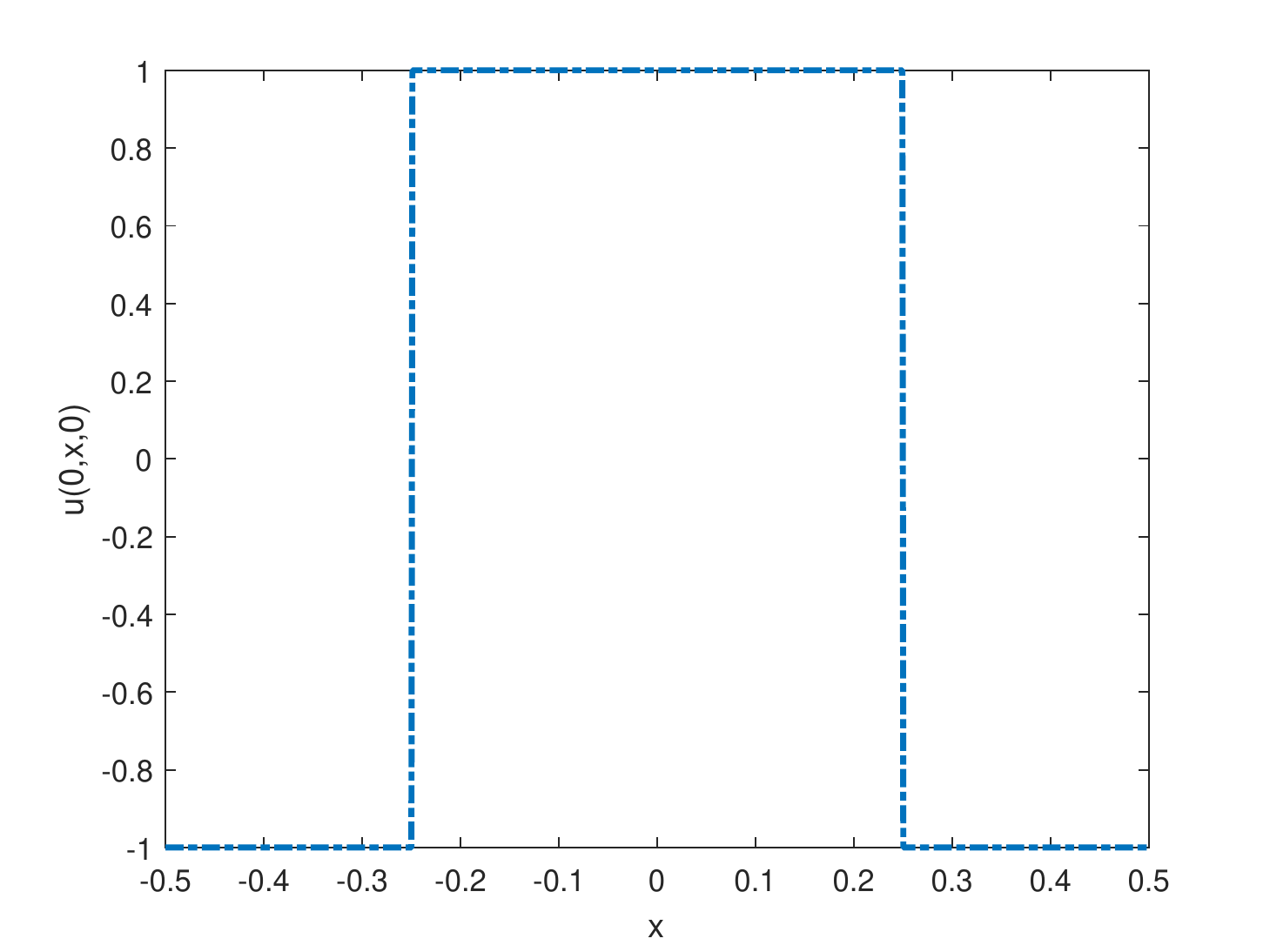}}
    \vspace{-0.2cm}
\caption{Initial configuration of the bubble.
Left: the surface-project view; right: the cross-section view at $y=0$.}
\label{fig5.7}
\end{figure}

The sIFRK(2,2) scheme is adopted and the parameters of the space-time mesh are set to be $h=1/1024$ and $\tau=0.01$.
\figurename~\ref{fig5.8} presents the evolutions of the bubble at times $t=50$,  $100$, $150$, $200$, $250$ and $300$, respectively,
and the left graph in \figurename~\ref{fig5.9} gives the corresponding cross-section views with $y=0$.
The right graph in \figurename~\ref{fig5.9} presents the evolution of the energy, which is  monotonically decreasing.
The bubble shrinks smaller and smaller during the evolution and finally vanishes at about $t=310$.

\begin{figure}[!ht]
  \centerline{
   \includegraphics[width=0.32\textwidth]{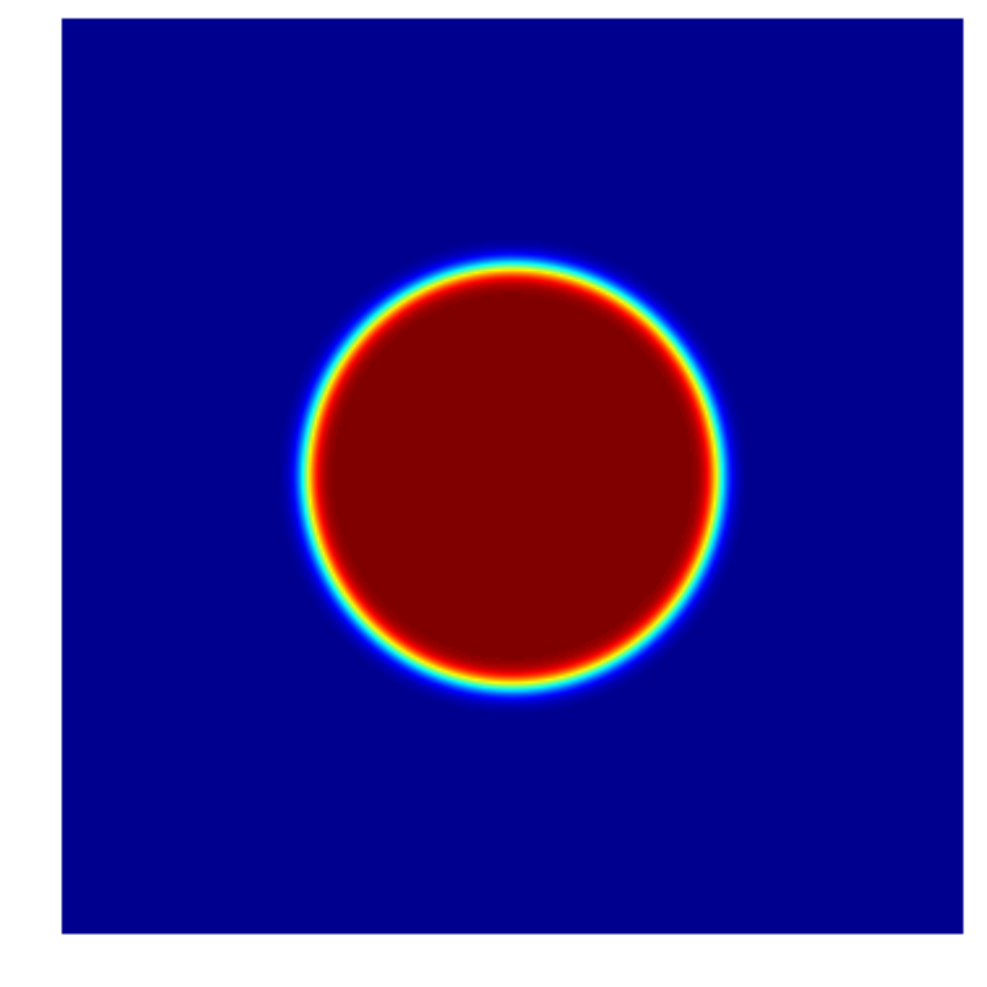}\hspace{-0.1cm}
   \includegraphics[width=0.32\textwidth]{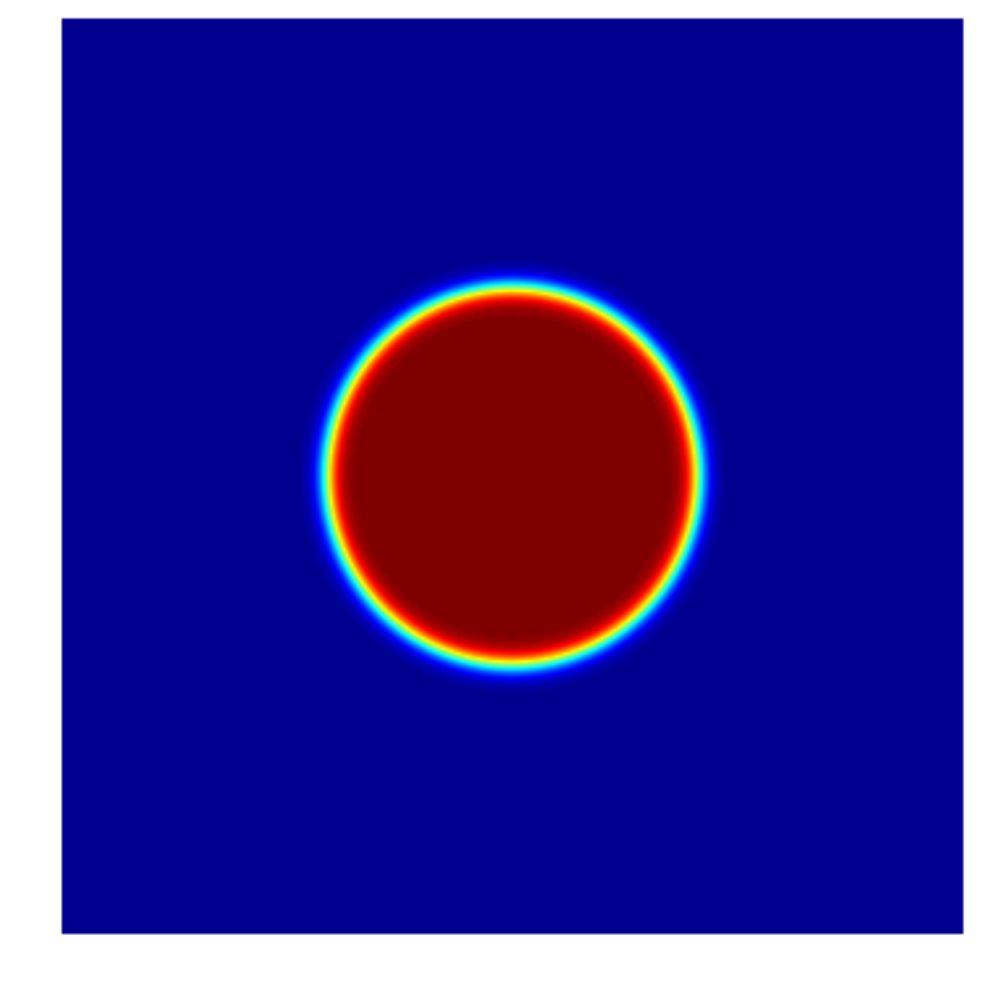}\hspace{-0.1cm}
  \includegraphics[width=0.32\textwidth]{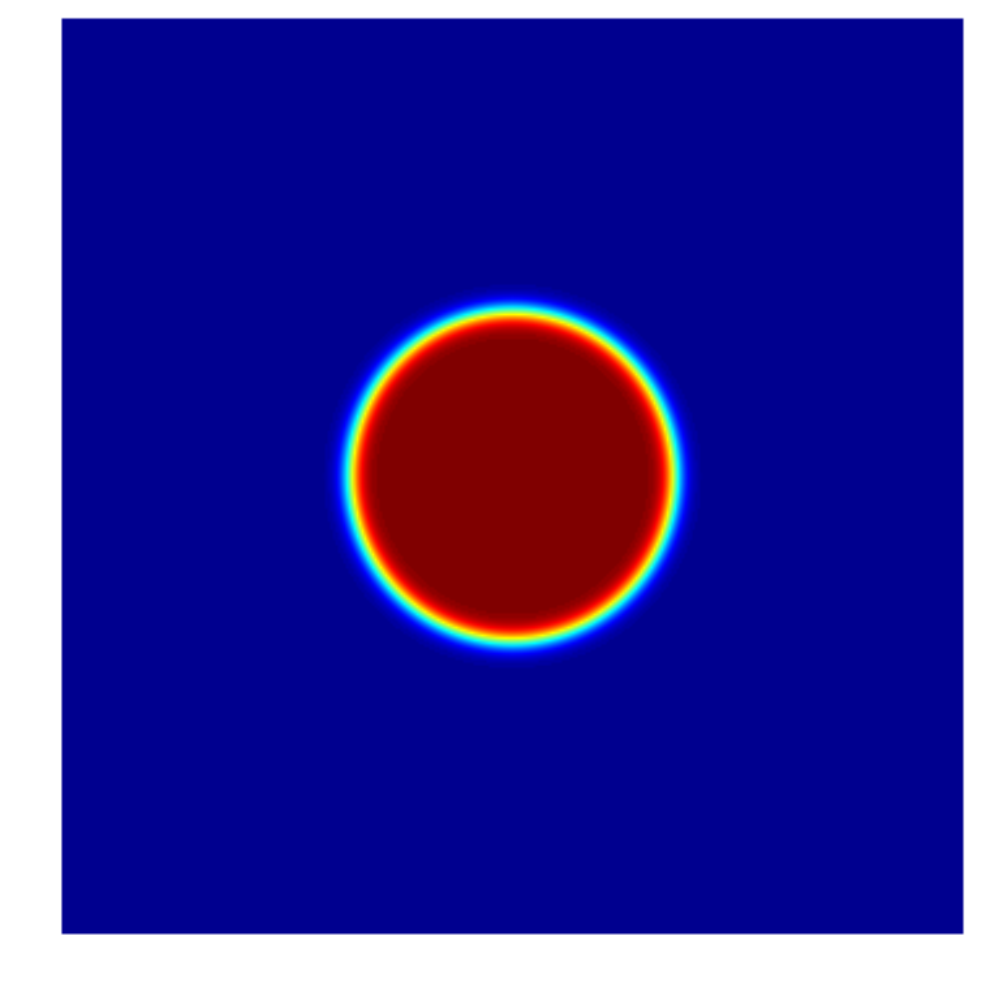}}
   \centerline{
   \includegraphics[width=0.32\textwidth]{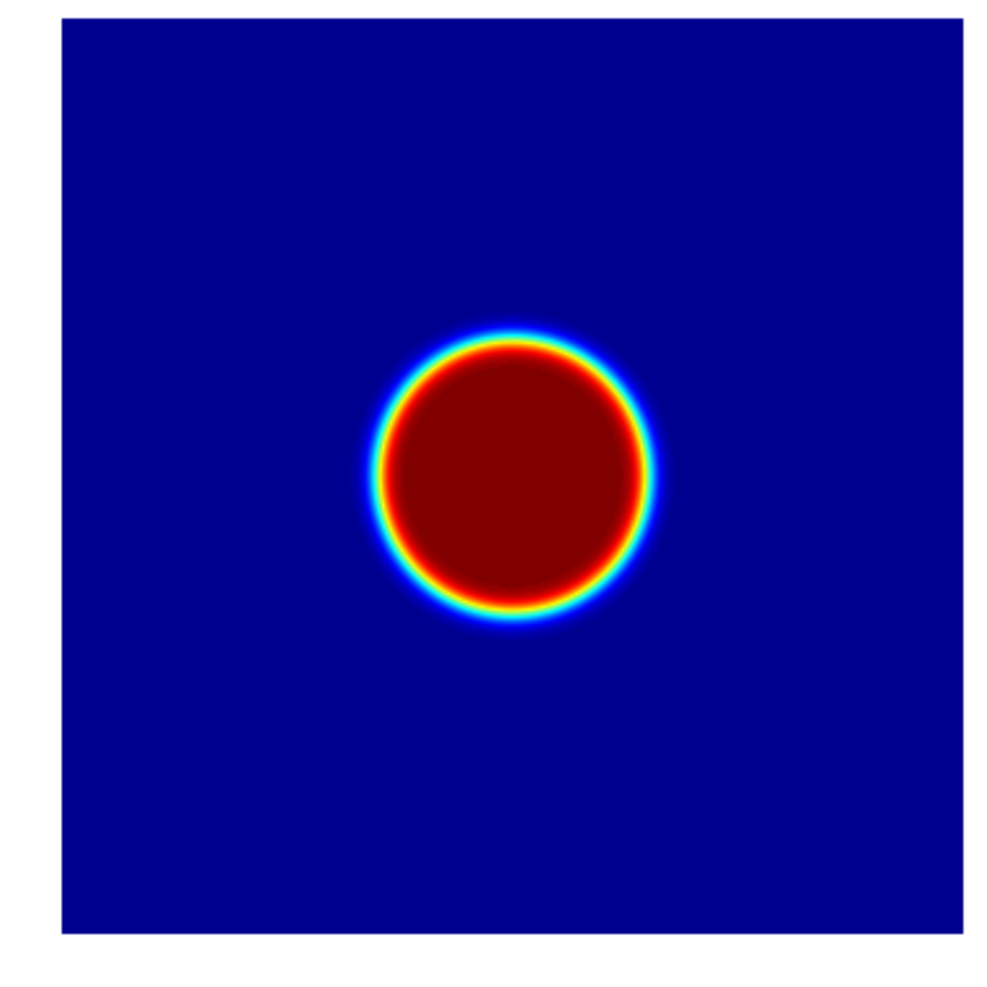}\hspace{-0.1cm}
  \includegraphics[width=0.32\textwidth]{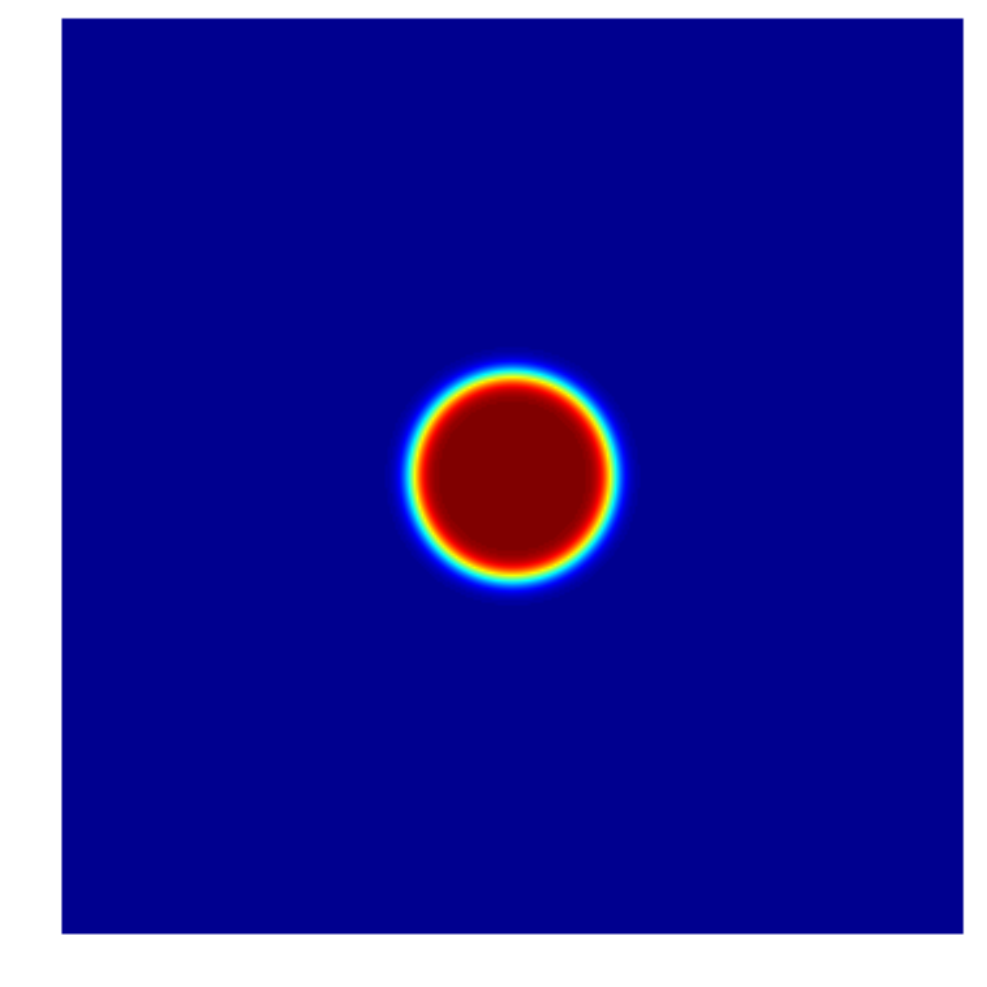}\hspace{-0.1cm}
   \includegraphics[width=0.32\textwidth]{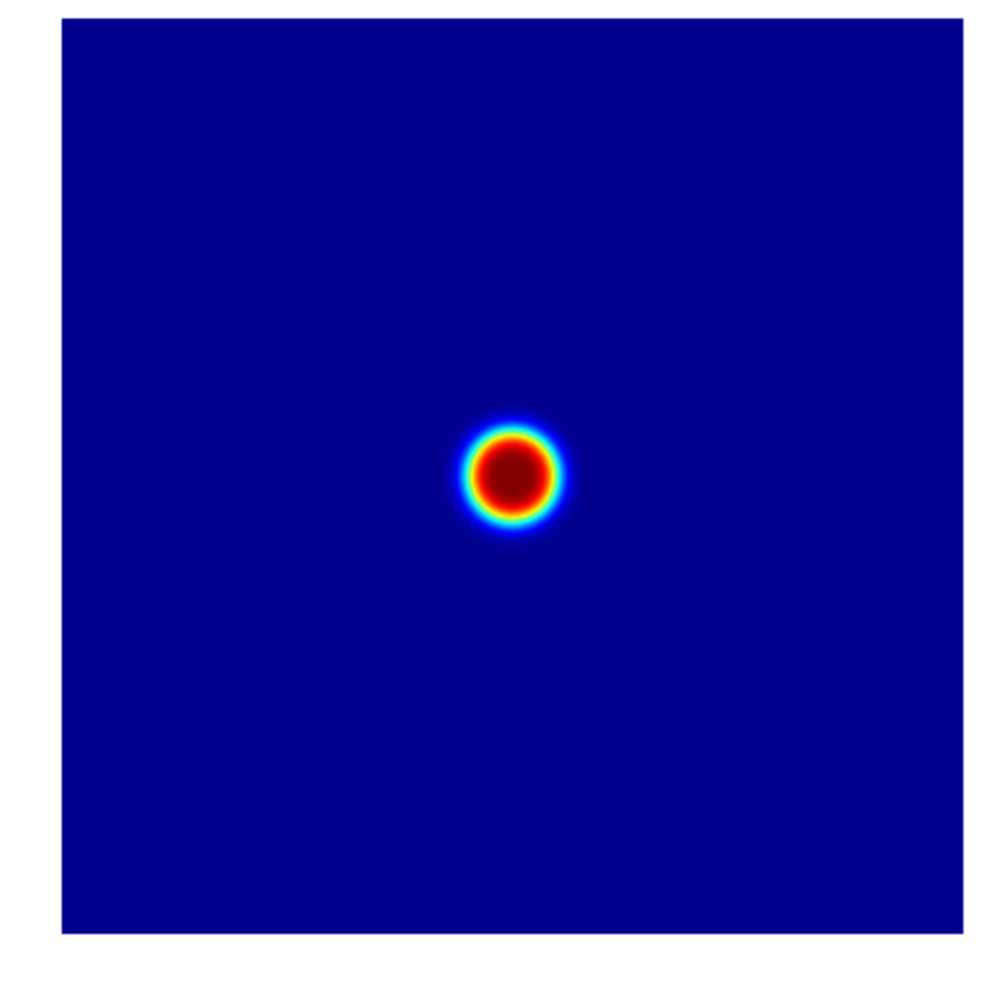}}
    \vspace{-0.2cm}
\caption{Evolution of the shrinking bubble obtained by using sIFRK(2,2) with \textcolor{black}{$\tau=0.01$}.
From left to right {and from top to bottom}: $t=50$,  $100$, $150$, $200$, $250$ and $300$.}
\label{fig5.8}
\end{figure}

\begin{figure}[!ht]
  \centerline{
  \includegraphics[width=0.49\textwidth]{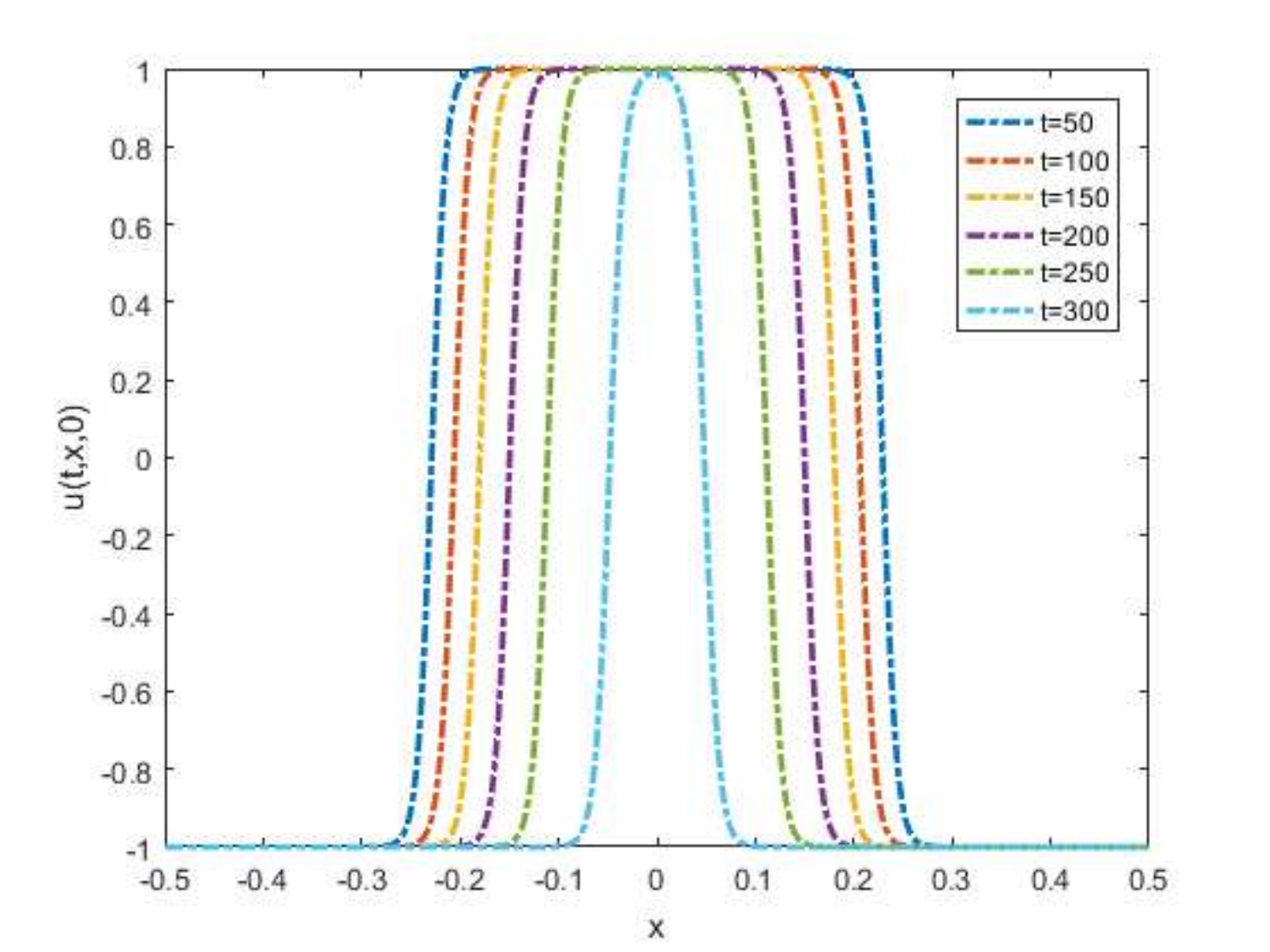}\hspace{-0.5cm}
   \includegraphics[width=0.49\textwidth]{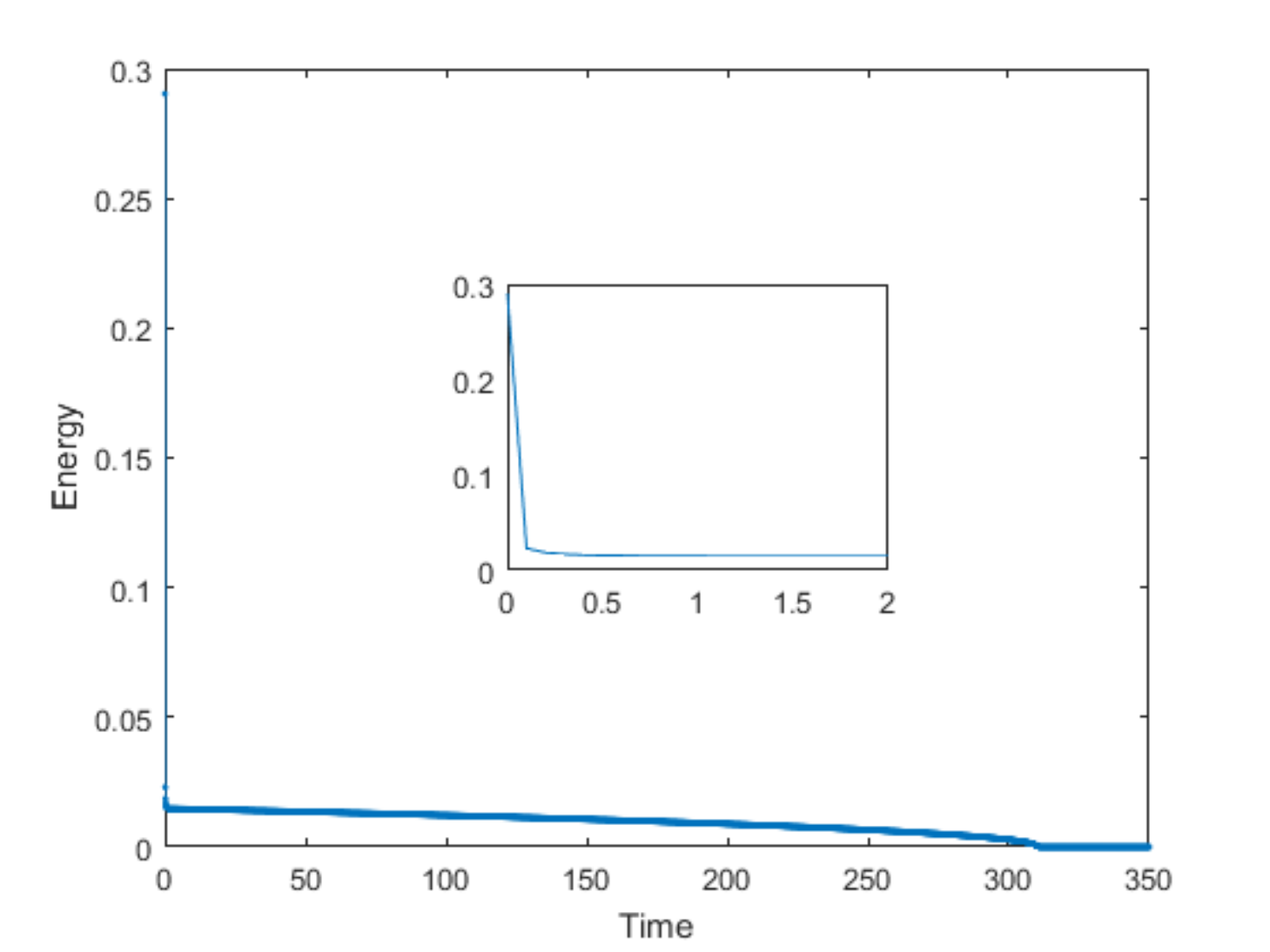}}
   \vspace{-0.2cm}
\caption{The cross-section view with $y=0$ (left) corresponding to \figurename~\ref{fig5.8}  and the evolution of the energy (right).}
\label{fig5.9}
\end{figure}

Next, we numerically show that the  SSP-sIFRK(2,2) scheme \eqref{ssps2}  is  not unconditionally MBP-preserving
in response to the discussion in Remark \ref{rem22}.
To this end, we consider the Allen-Cahn equation (\ref{eq5.4}) with $\epsilon =0.01$ and  \eqref{eq5.6} subject to the periodic boundary condition.
The initial data is generated by a set of random numbers ranging from $-0.9$ to $0.9$ uniformly on the spatial mesh with $h=1/256$.
The time-step size is set to be $\tau=0.1$, which is 10 times larger than that in the previous experiments.
We first run the simulation using SSP-sIFRK(2,2) taking the following form
\begin{subequations}
\begin{align*}
u^{(1)} & = e^{\tau\mathcal{L}^h_\kappa} (u^n+\tau\mathcal{N}[u^n]), \\
u^{n+1} & = \frac{1}{2} e^{\tau\mathcal{L}^h_\kappa} u^n + \frac{1}{2} (u^{(1)}+\tau\mathcal{N}[u^{(1)}]).
\end{align*}
\end{subequations}
Then we  re-run the simulation using sIFRK(2,2) with the same initial data.
Figs.~\ref{fig5.3} and \ref{fig5.4} present the evolutions of the phase structures
at $t=1$, $5$, $10$, $50$, $240$ and $440$ produced by SSP-sIFRK(2,2) and sIFRK(2,2),
respectively, which show that
the simulation results start to differ very soon although {we use the same initial data and space-time parameters.}
Figs.~\ref{fig5.5} and \ref{fig5.6} present the evolutions of
the corresponding supremum norms and the energies for SSP-sIFRK(2,2) and sIFRK(2,2), respectively.
It is observed that the MBP is preserved perfectly by sIFRK(2,2) for all time.
However, the solution by SSP-sIFRK(2,2) has the supremum norm beyond $1$ after $t=5.5$, which implies SSP-sIFRK(2,2) does not preserve the MBP in this case.

In addition, we  carry out the same experiment using sIFRK(2,2), but with different time-step sizes, $\tau=0.05$ and $0.01$.
We found that both  simulated processes of the phase transition are almost identical to that  illustrated in \figurename~\ref{fig5.4}.
The corresponding evolutions of the supremum norms and the energies are also given and compared with those produced with $\tau=0.1$ in \figurename~\ref{fig5.6}, which shows very small differences between them. These observations partly imply that the error constant in Theorem \ref{errthm} does not change much when those different time-step sizes are adopted, and the performance of sIFRK(2,2) is still satisfactory, in terms of accuracy and efficiency, for practical simulations with moderately large time-step sizes.

\begin{figure}[!ht]
  \centerline{
   \includegraphics[width=0.32\textwidth]{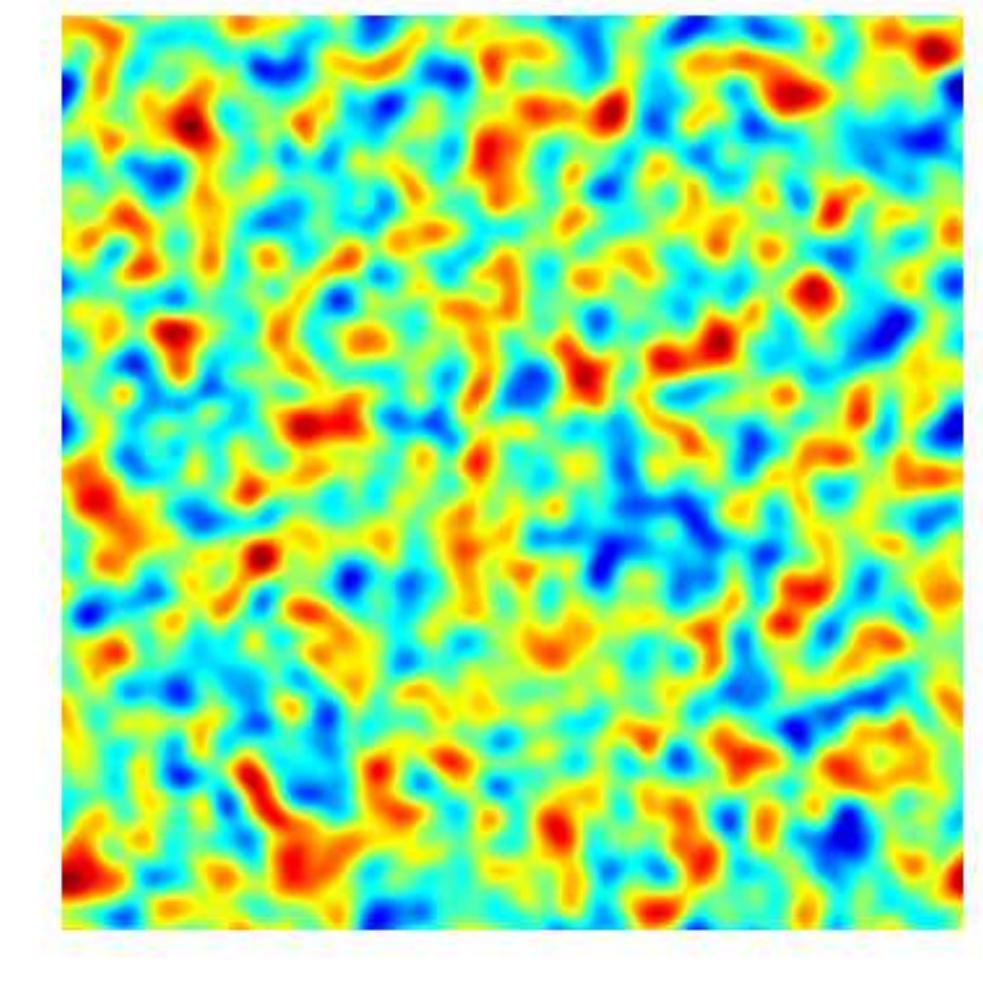}\hspace{-0.1cm}
   \includegraphics[width=0.32\textwidth]{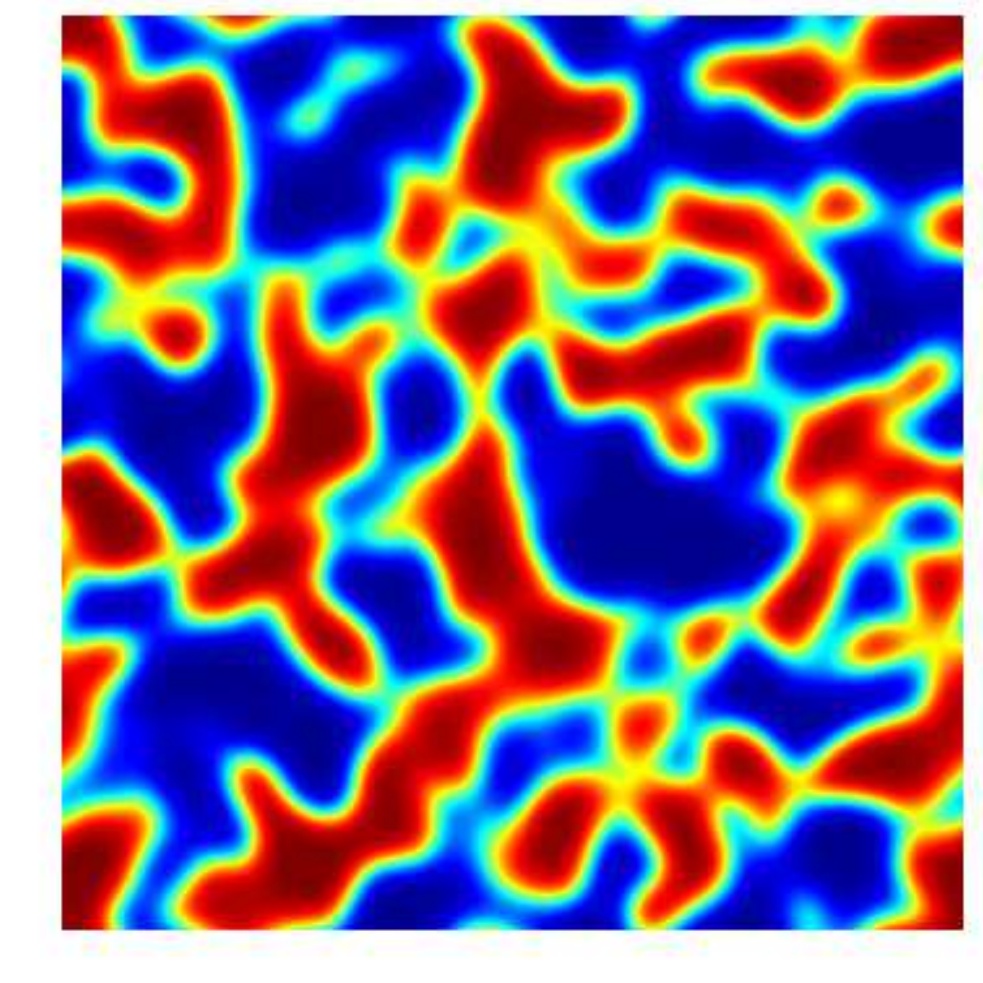}\hspace{-0.1cm}
   \includegraphics[width=0.32\textwidth]{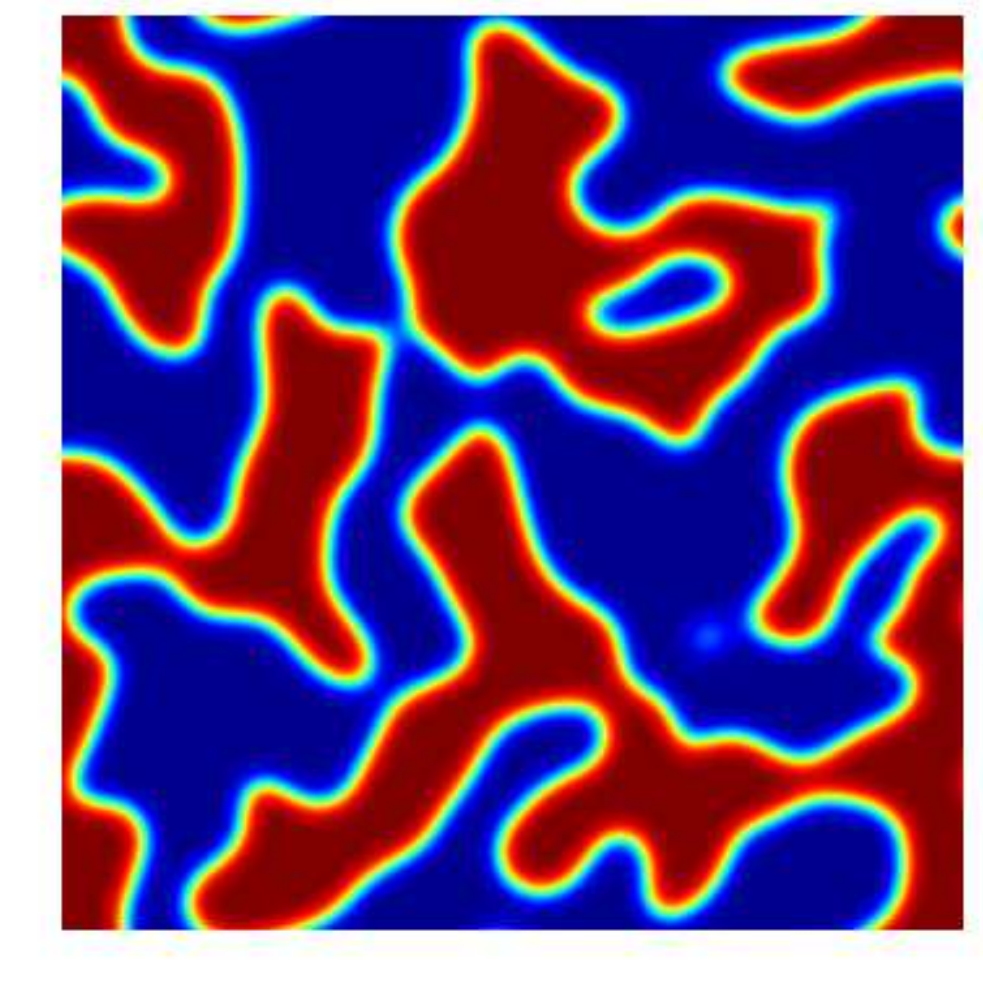}}
   \centerline{
   \includegraphics[width=0.32\textwidth]{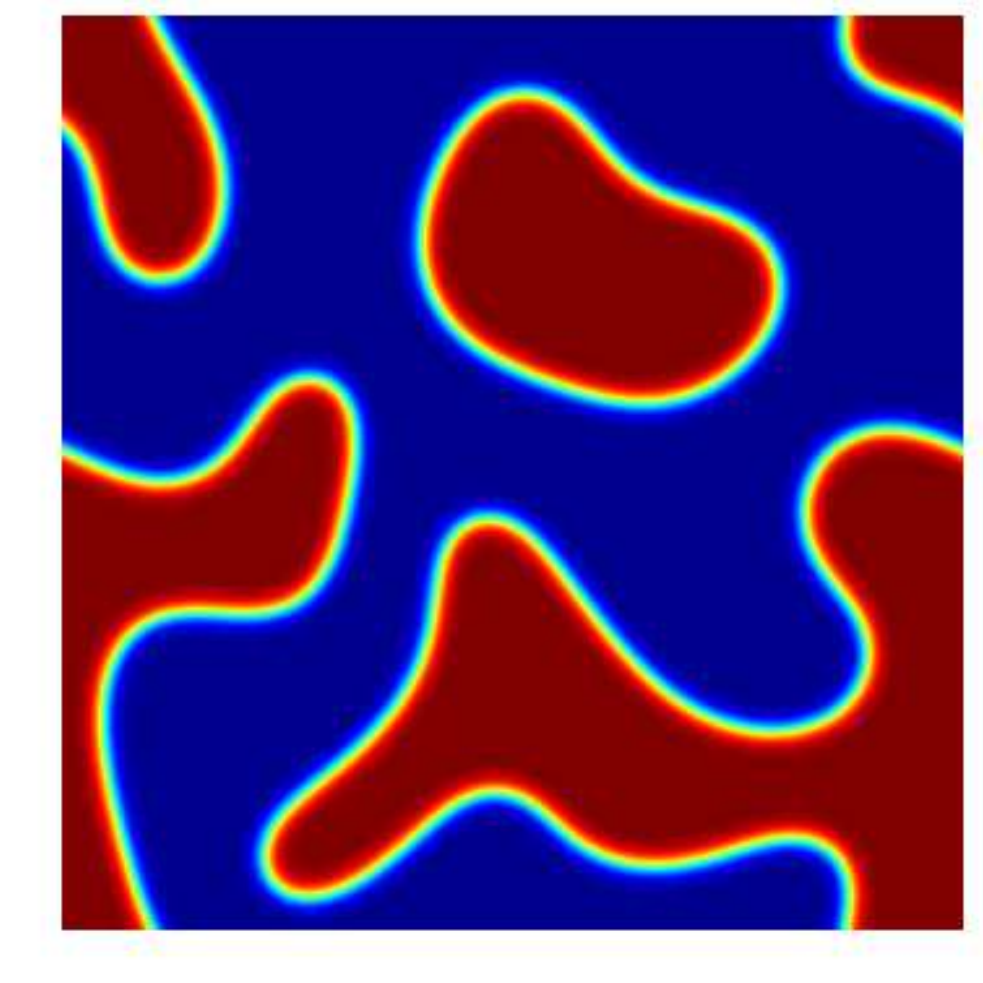}\hspace{-0.1cm}
   \includegraphics[width=0.32\textwidth]{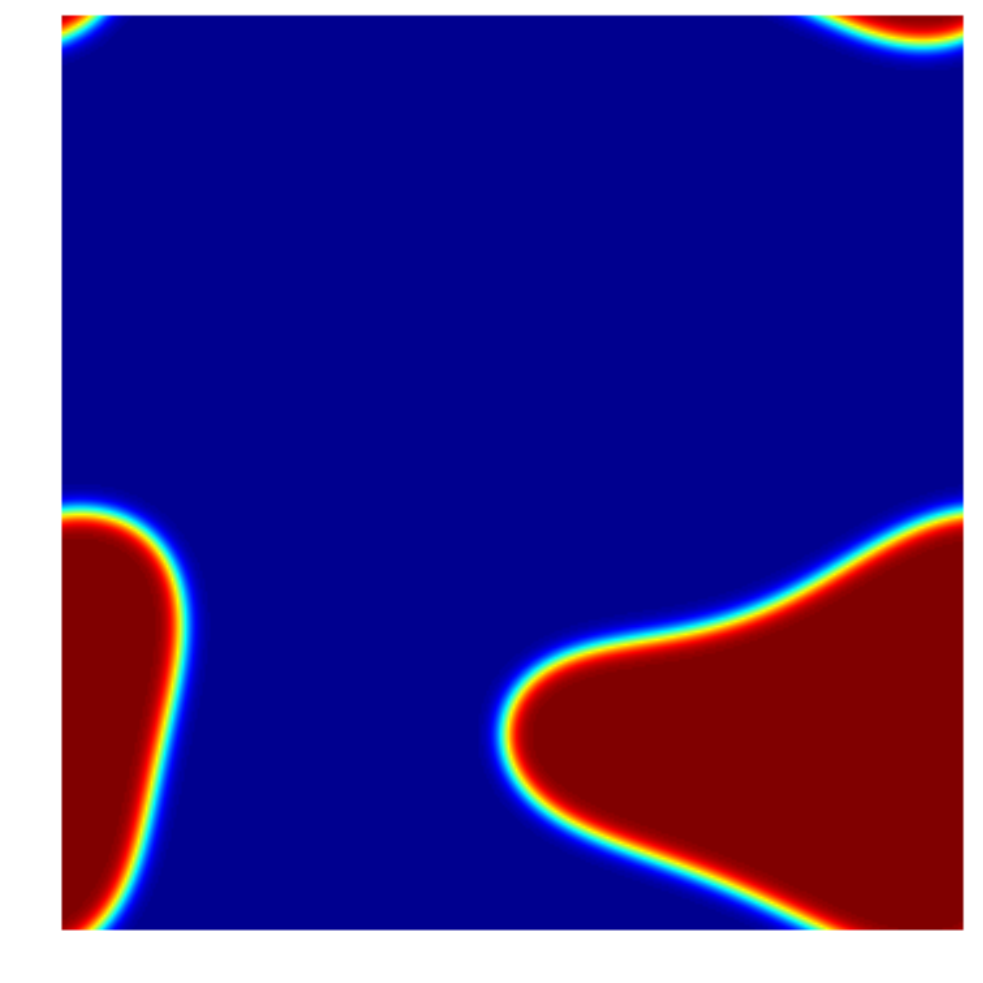}\hspace{-0.1cm}
   \includegraphics[width=0.32\textwidth]{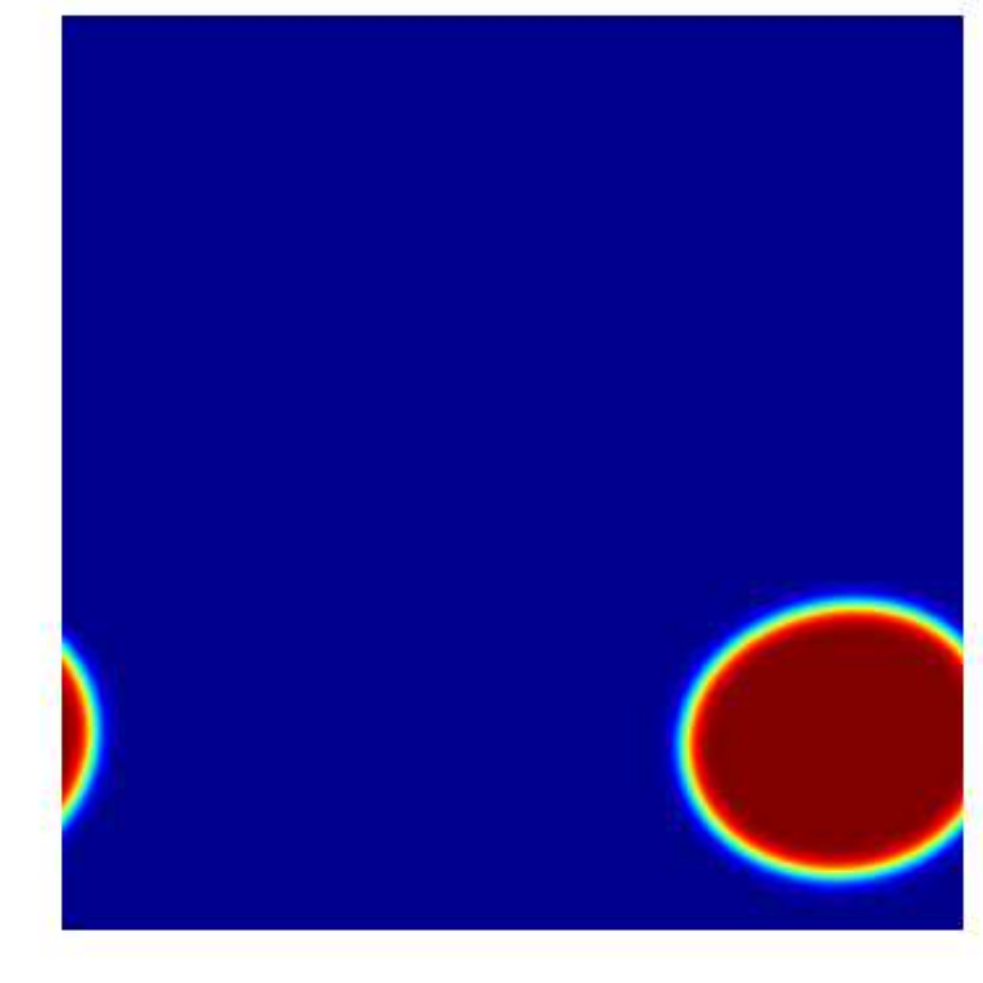}}
   \vspace{-0.2cm}
\caption{Evolution of the phase structure obtained by using SSP-sIFRK(2,2) {with $\tau=0.1$}.
From left to right and from top to bottom: $t=1$, $5$, $10$, $50$, $240$ and $440$.}
\label{fig5.3}
\end{figure}

\begin{figure}[!ht]
  \centerline{
   \includegraphics[width=0.49\textwidth]{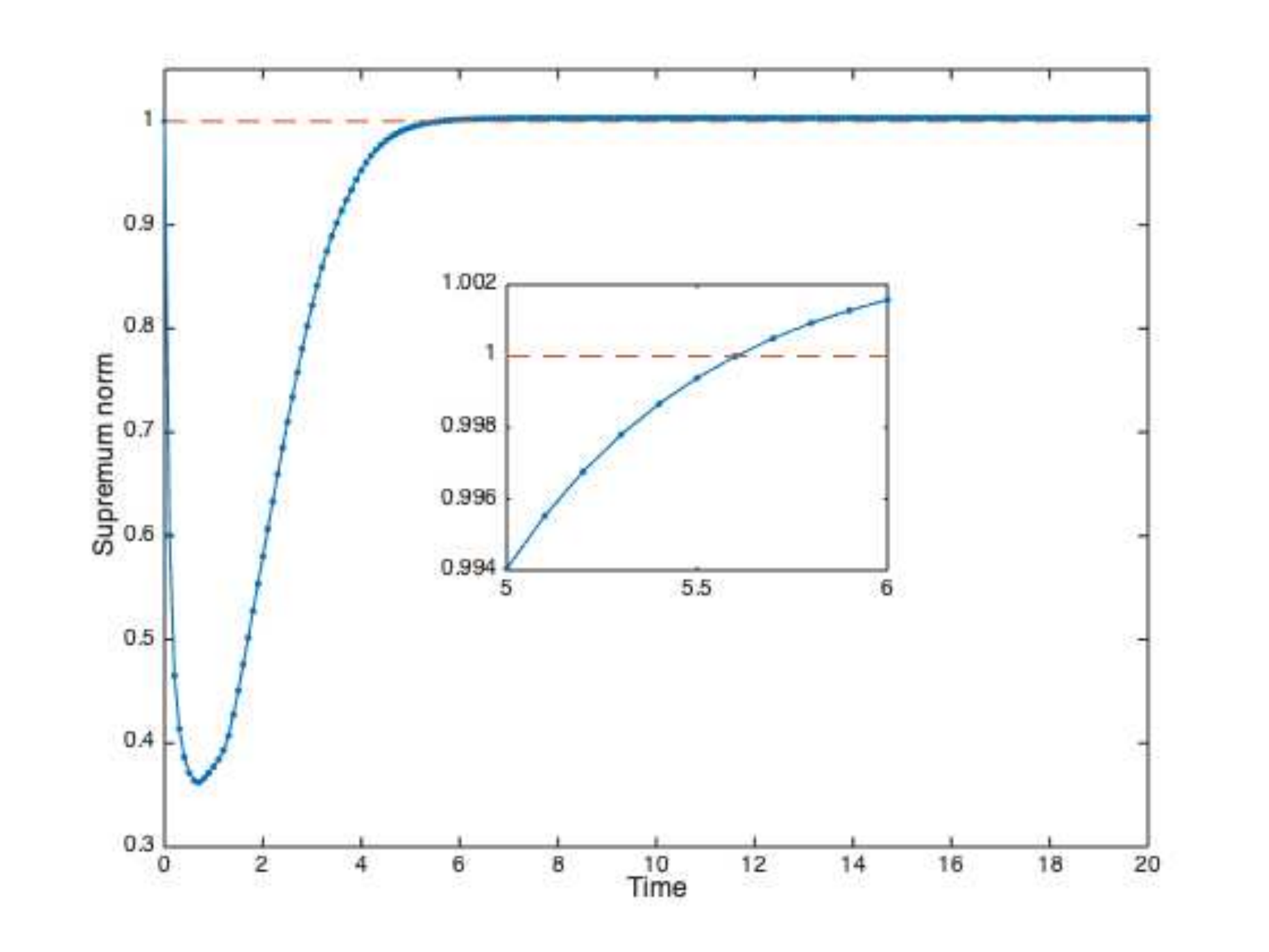}\hspace{-0.5cm}
   \includegraphics[width=0.49\textwidth]{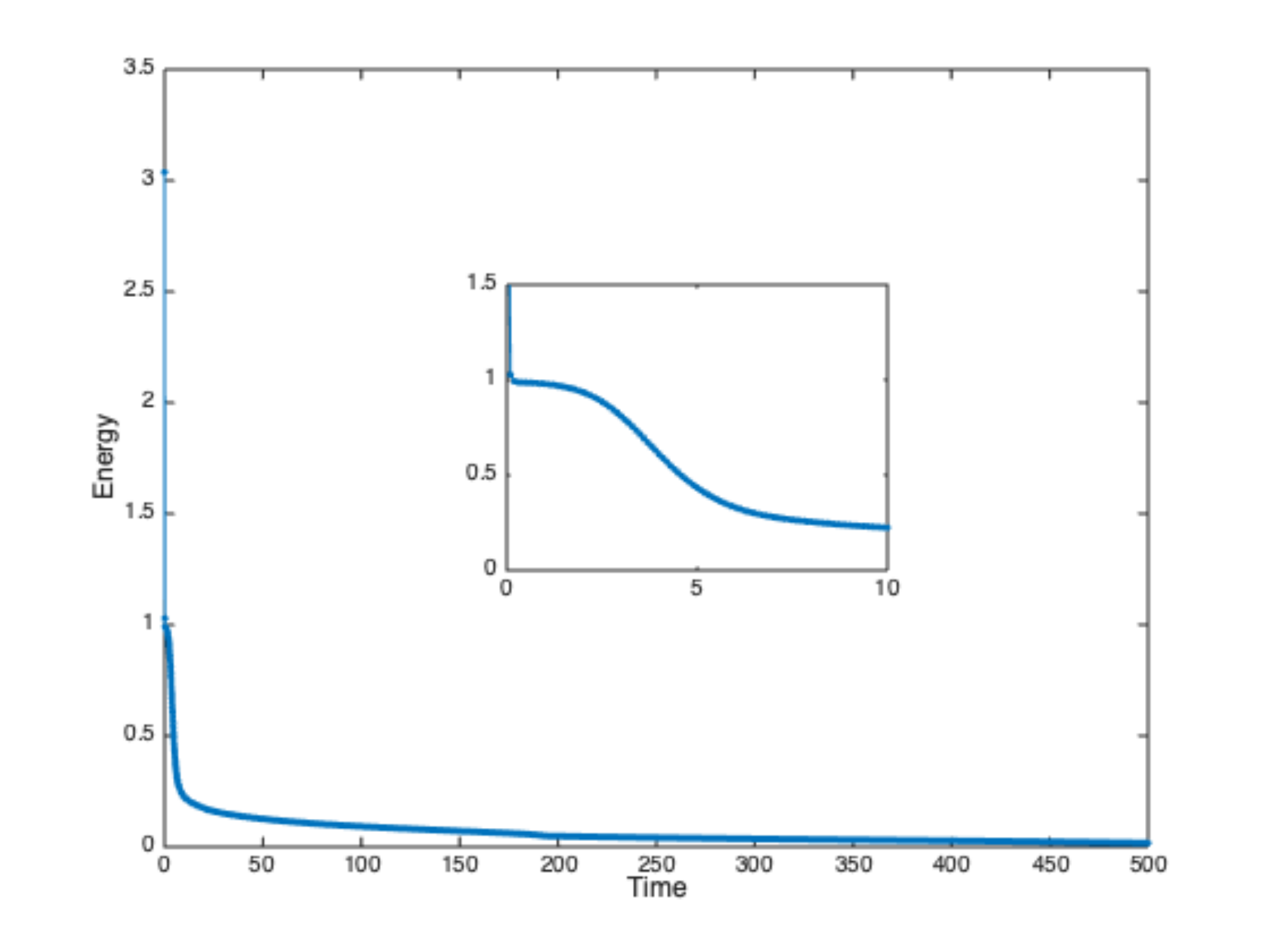}}
     \vspace{-0.2cm}
\caption{Evolutions of the supremum norm (left) and the energy (right) obtained by using SSP-sIFRK(2,2) {with $\tau=0.1$}.}
\label{fig5.5}
\end{figure}

\begin{figure}[!ht]
  \centerline{
   \includegraphics[width=0.32\textwidth]{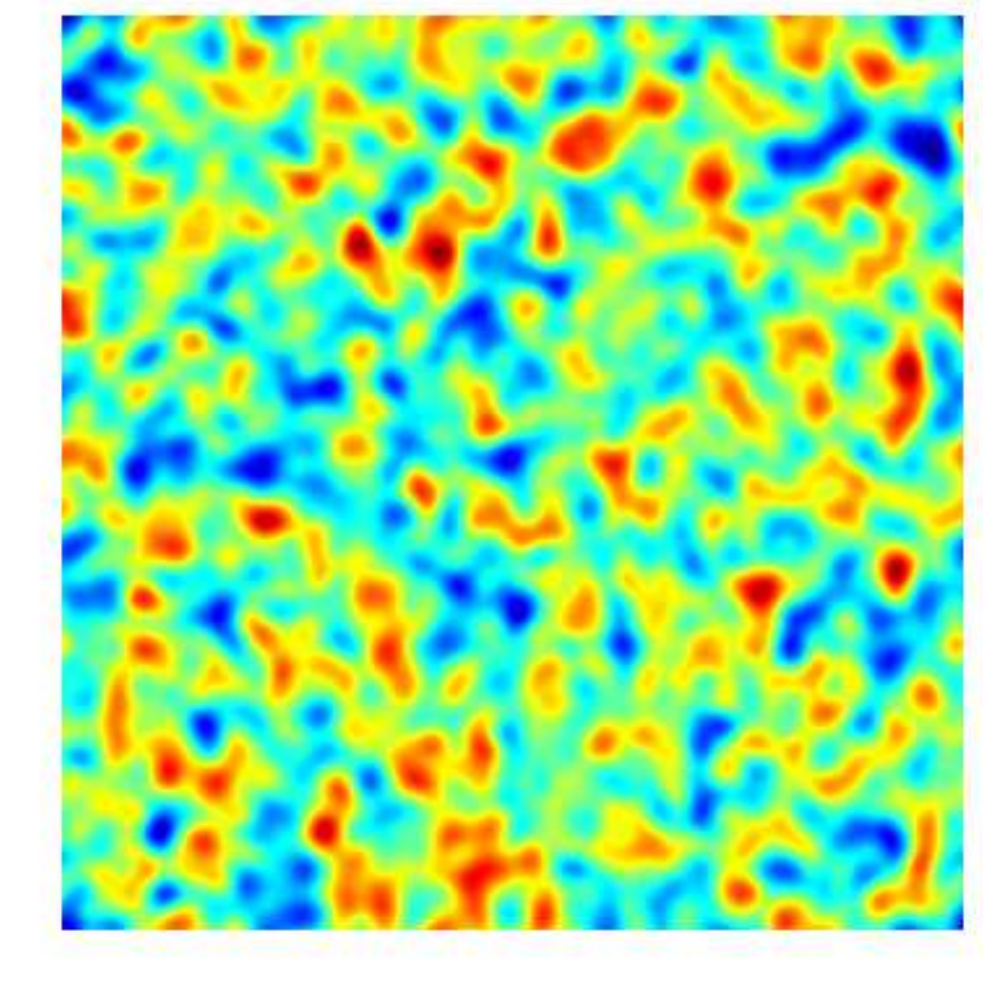}\hspace{-0.1cm}
   \includegraphics[width=0.32\textwidth]{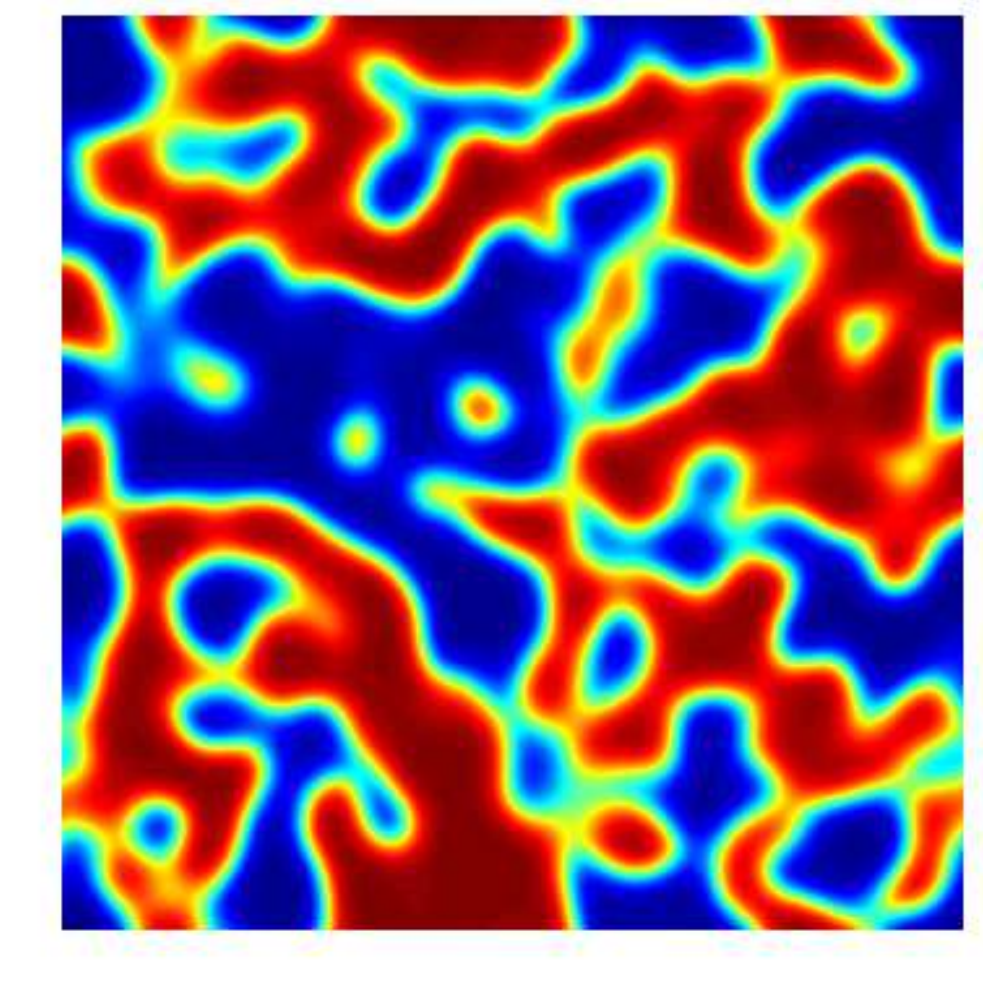}\hspace{-0.1cm}
   \includegraphics[width=0.32\textwidth]{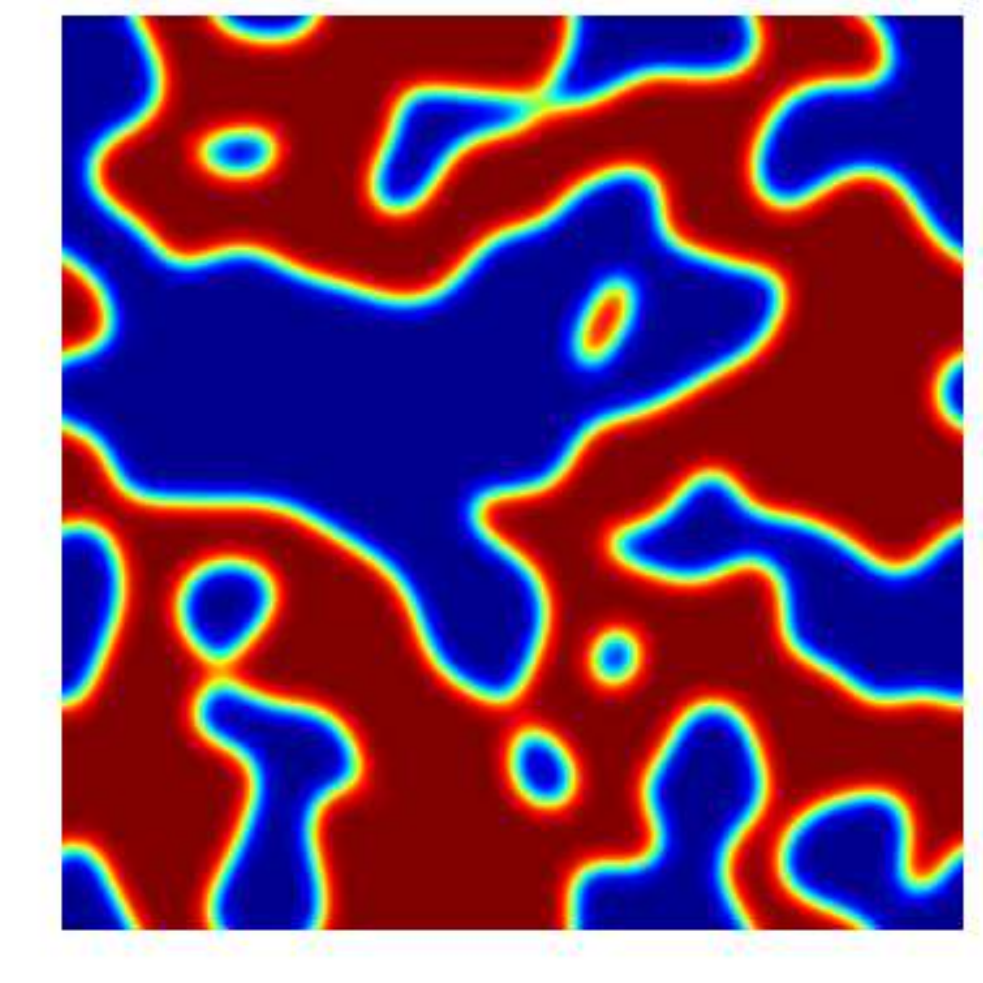}}
   \centerline{
    \includegraphics[width=0.32\textwidth]{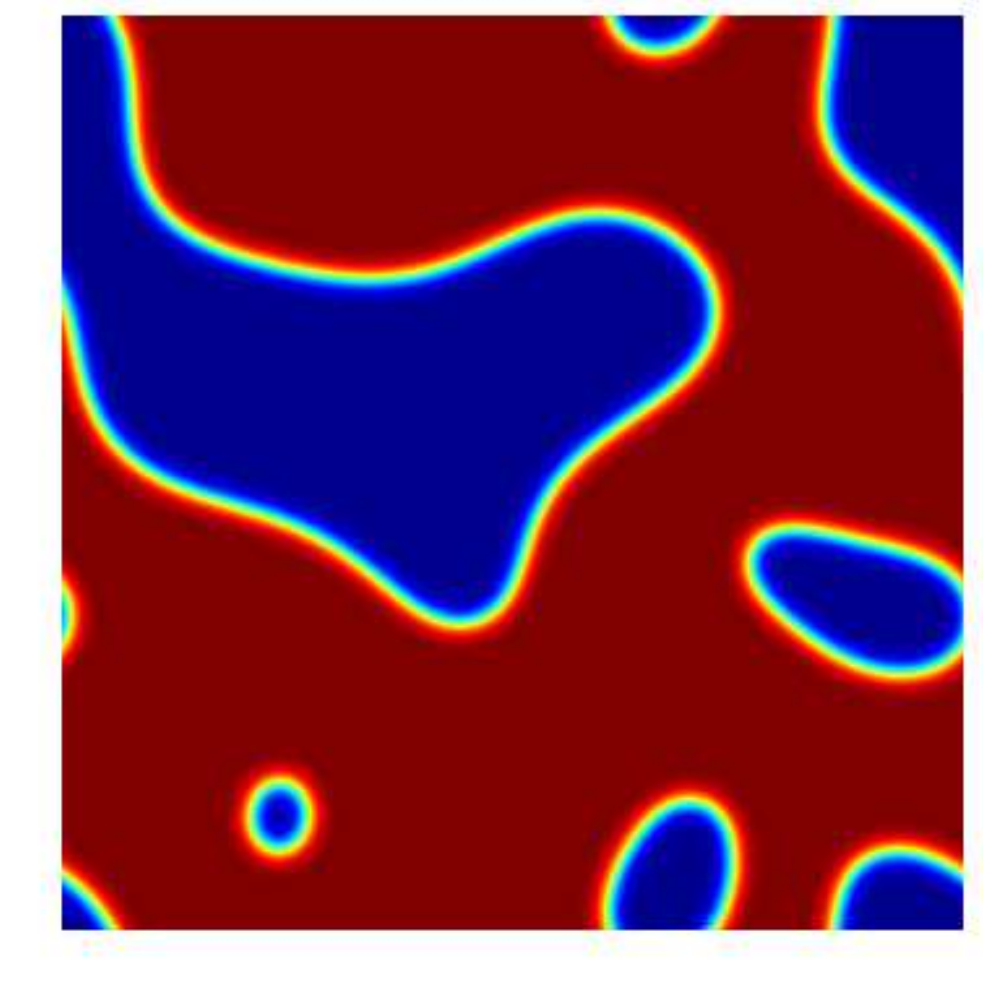}\hspace{-0.1cm}
   \includegraphics[width=0.32\textwidth]{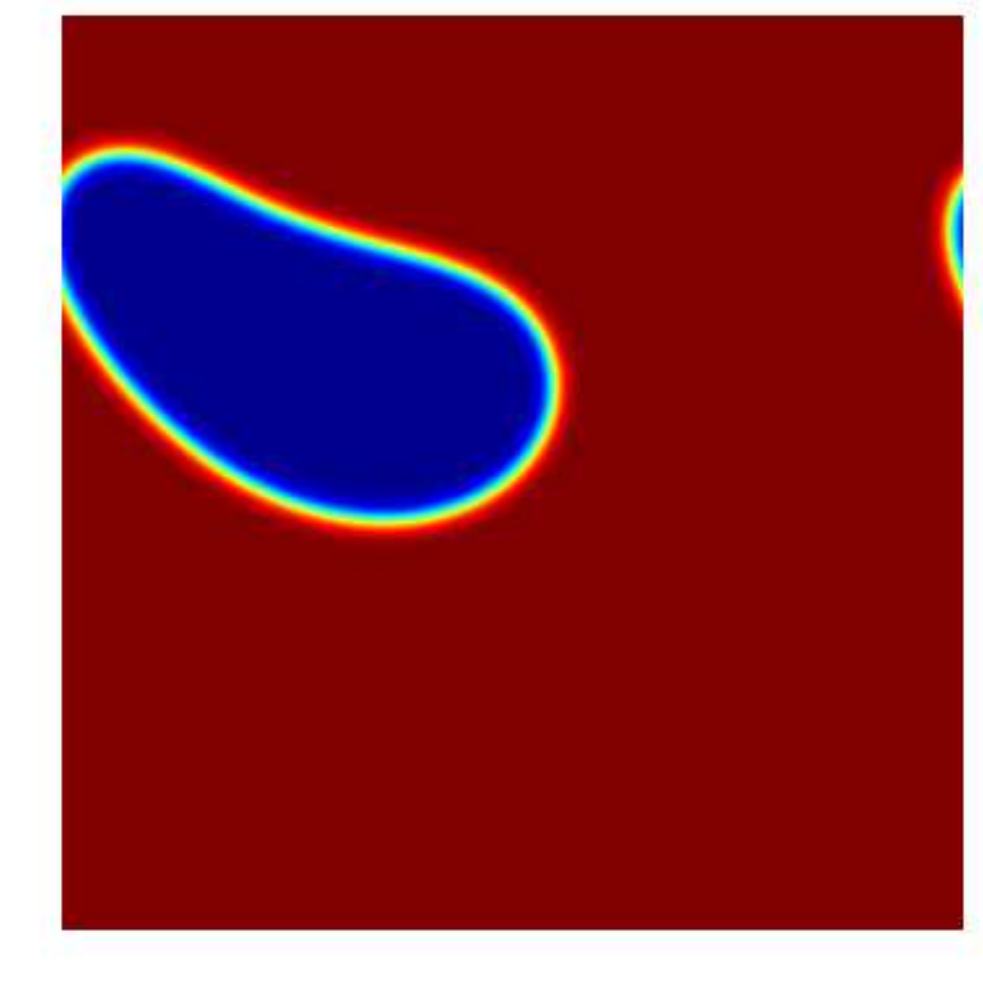}\hspace{-0.1cm}
   \includegraphics[width=0.32\textwidth]{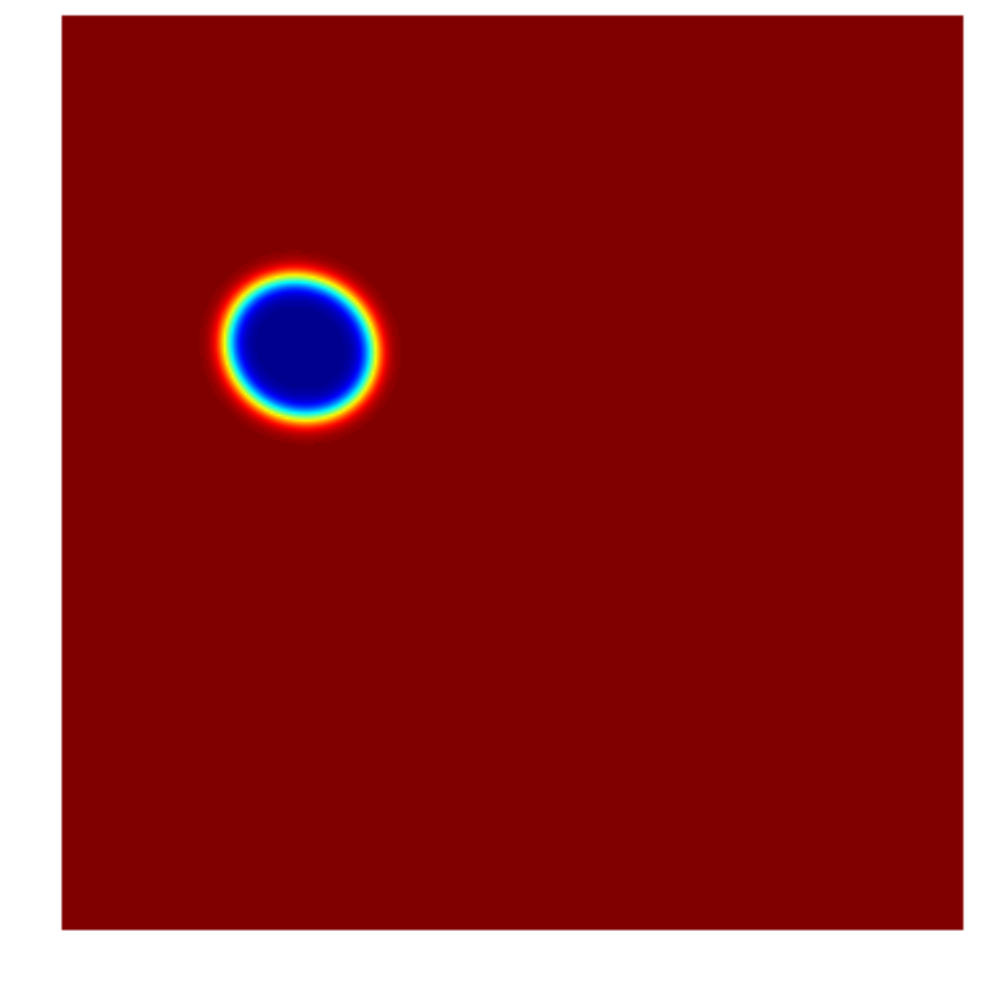}}
     \vspace{-0.2cm}
\caption{Evolution of the phase structure obtained by using sIFRK(2,2) {with $\tau=0.1$}.
From left to right and from top to bottom: $t=1$, $5$, $10$, $50$, $240$ and $440$.}
\label{fig5.4}
\end{figure}

\begin{figure}[!ht]
 \centerline{
   \includegraphics[width=0.49\textwidth]{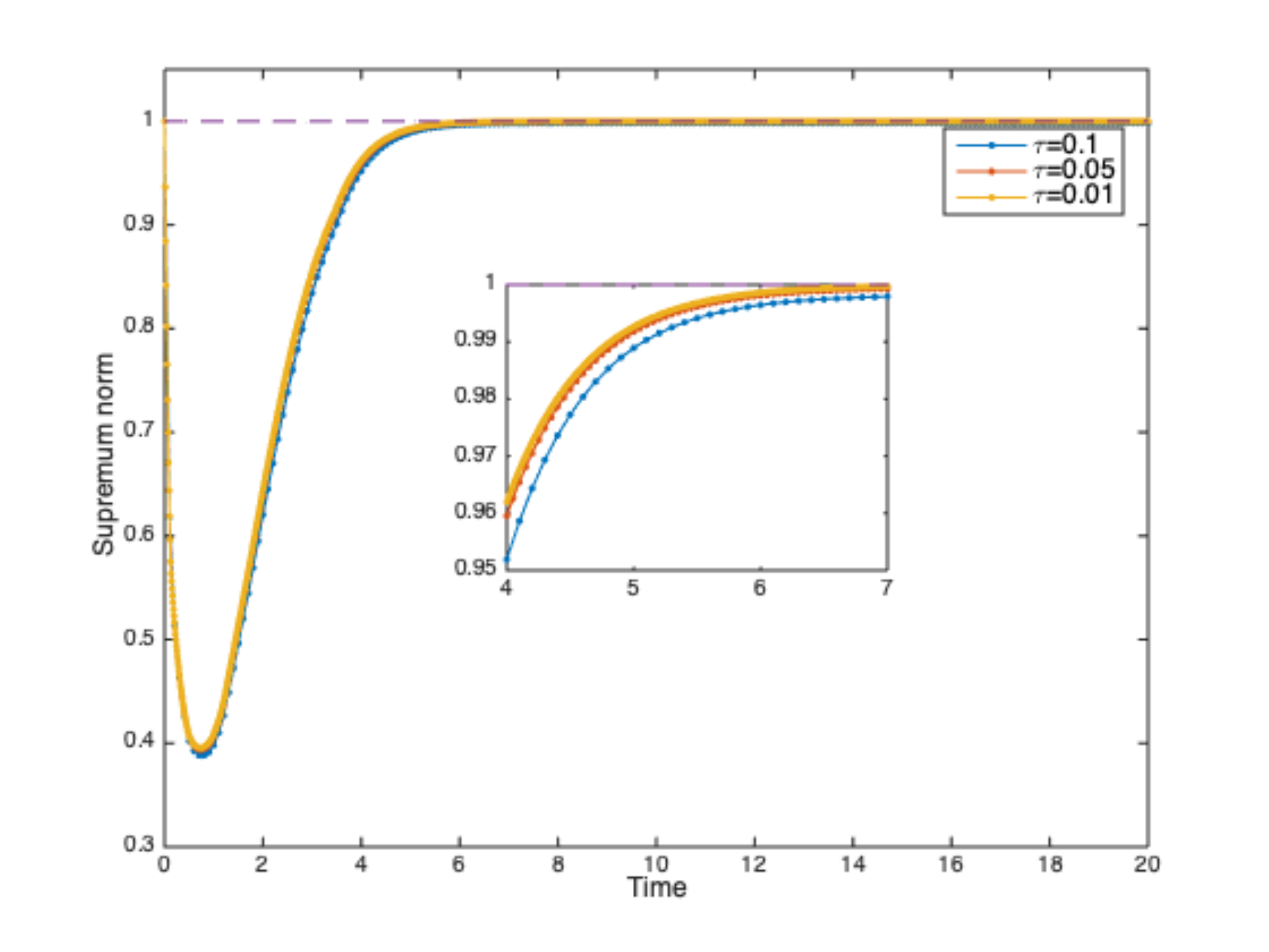}\hspace{-0.5cm}
   \includegraphics[width=0.49\textwidth]{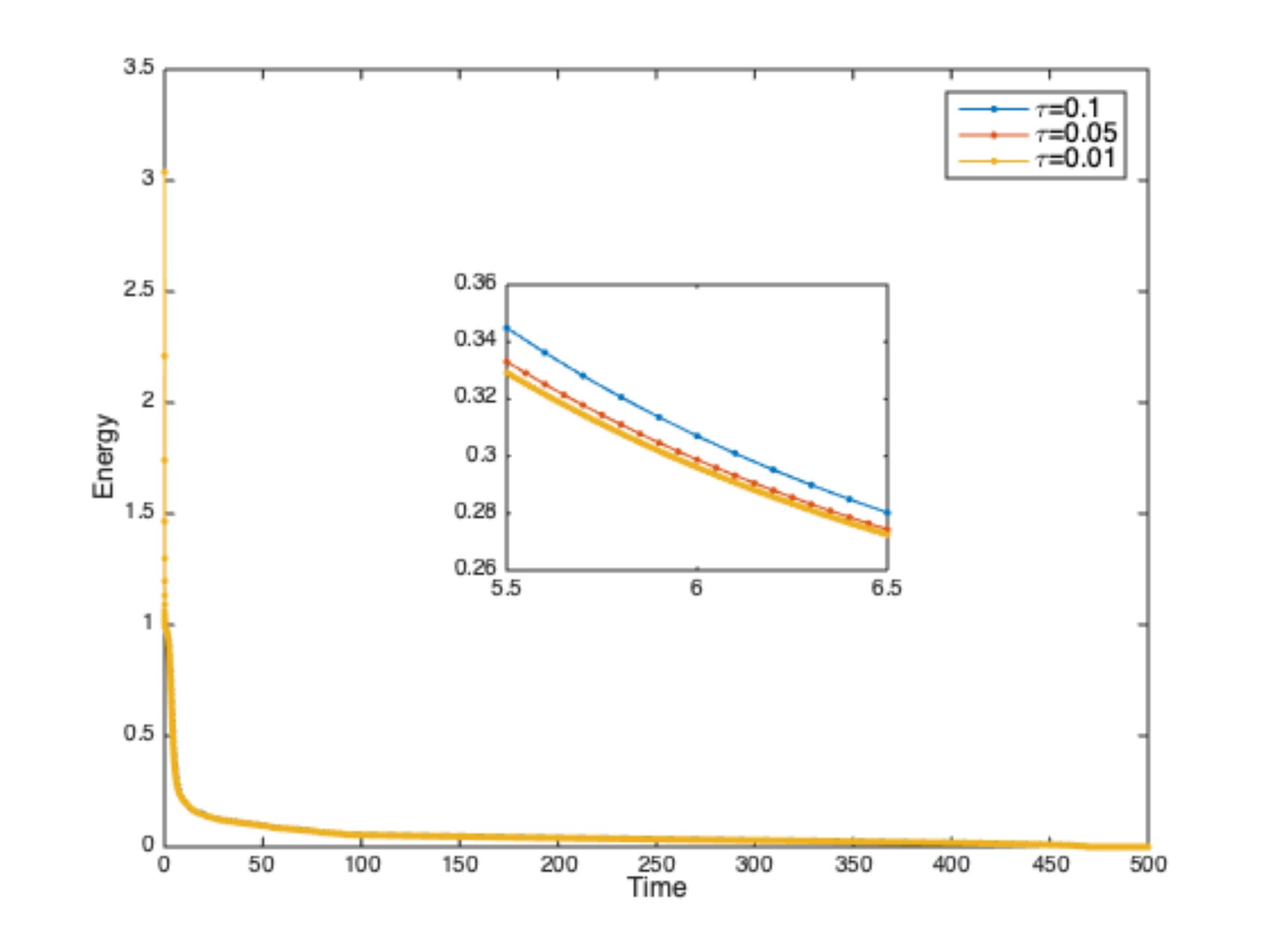}}
     \vspace{-0.2cm}
\caption{Evolutions of the supremum norm (left) and the energy (right) obtained by using sIFRK(2,2) with $\tau=0.1$, $0.05$ and $0.01$.}
\label{fig5.6}
\end{figure}

\subsection{Three-dimensional simulations}

The last experiment is devoted to  numerical simulation
for the 3D Allen-Cahn equation \eqref{eq5.4} with $\epsilon=0.01$ and \eqref{eq5.6}.
We take the domain $\Omega=(-0.5,0.5)^3$ with a uniform spatial mesh of size $h=1/256$,
and generate the initial data by the random numbers ranging from $-0.9$ to $0.9$ on the mesh.
The sIFRK(2,2) scheme is used for the simulation.
\figurename~\ref{fig3d-1} presents the evolutions of the 3D phase structures
at $t=1$, $5$, $10$, $50$, $240$ and $350$ with the time-step size $\tau=0.01$, respectively.
\figurename~\ref{fig3d-2} plots the evolutions of the supremum norm and the energy
of the numerical solutions with different time-step sizes $\tau=0.1$, $0.05$ and $0.01$.
The MBP is well-preserved and the energy decreases monotonically along the time as shown in \figurename~\ref{fig3d-2}.
In addition, we again observe that
there are only very small differences between the curves corresponding to different time-step sizes, which
implies that sIFRK(2,2) still perform very well for 3D simulations even with moderately large time-step sizes.

\begin{figure}[!ht]
  \centerline{\hspace{-0.2cm}
   \includegraphics[width=0.33\textwidth]{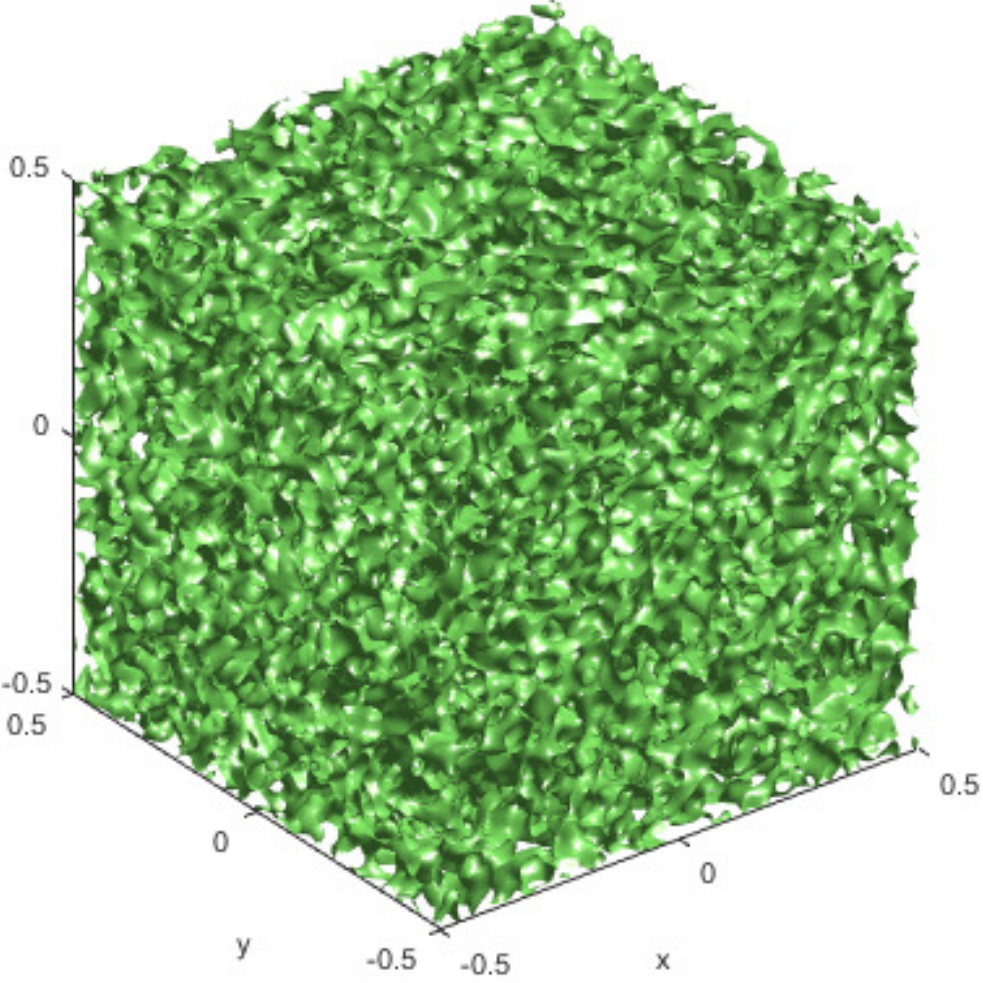}\hspace{-0.1cm}
   \includegraphics[width=0.33\textwidth]{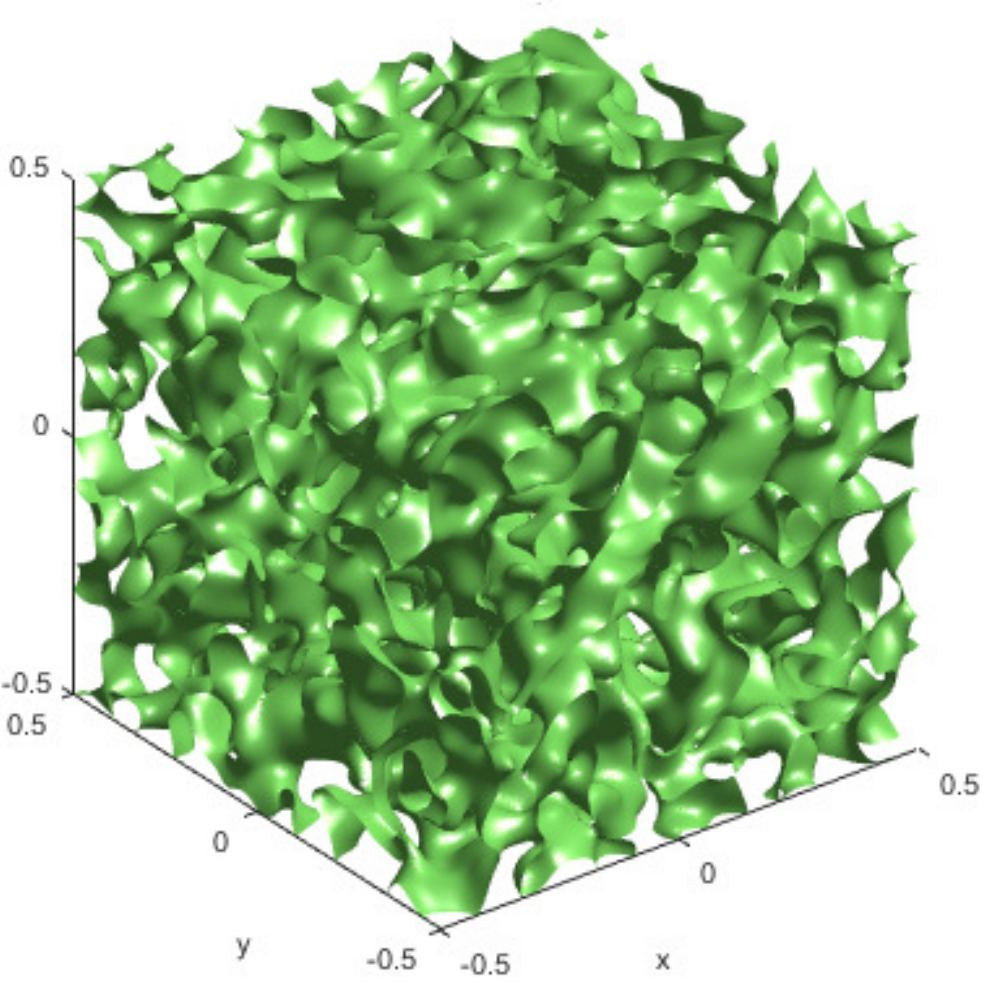}\hspace{-0.1cm}
   \includegraphics[width=0.33\textwidth]{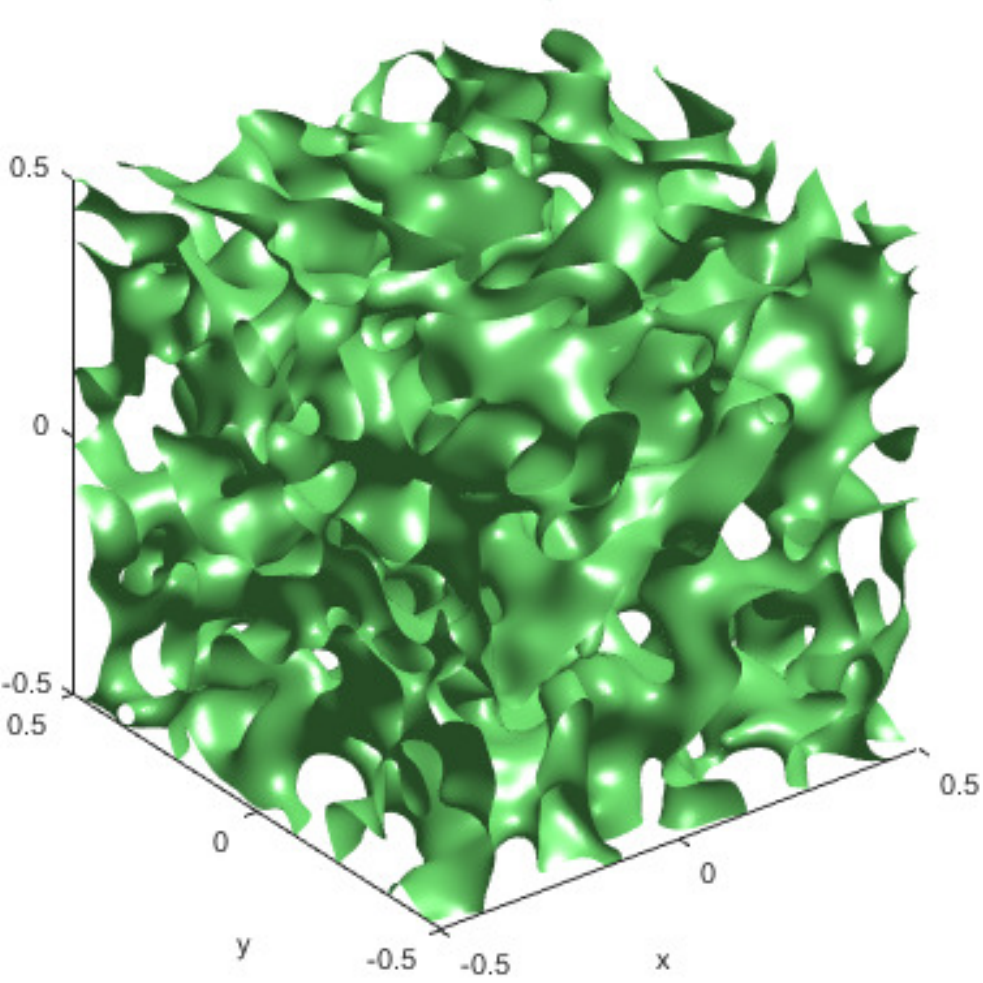}}
   \vspace{0.3cm}
   \centerline{\hspace{-0.2cm}
   \includegraphics[width=0.33\textwidth]{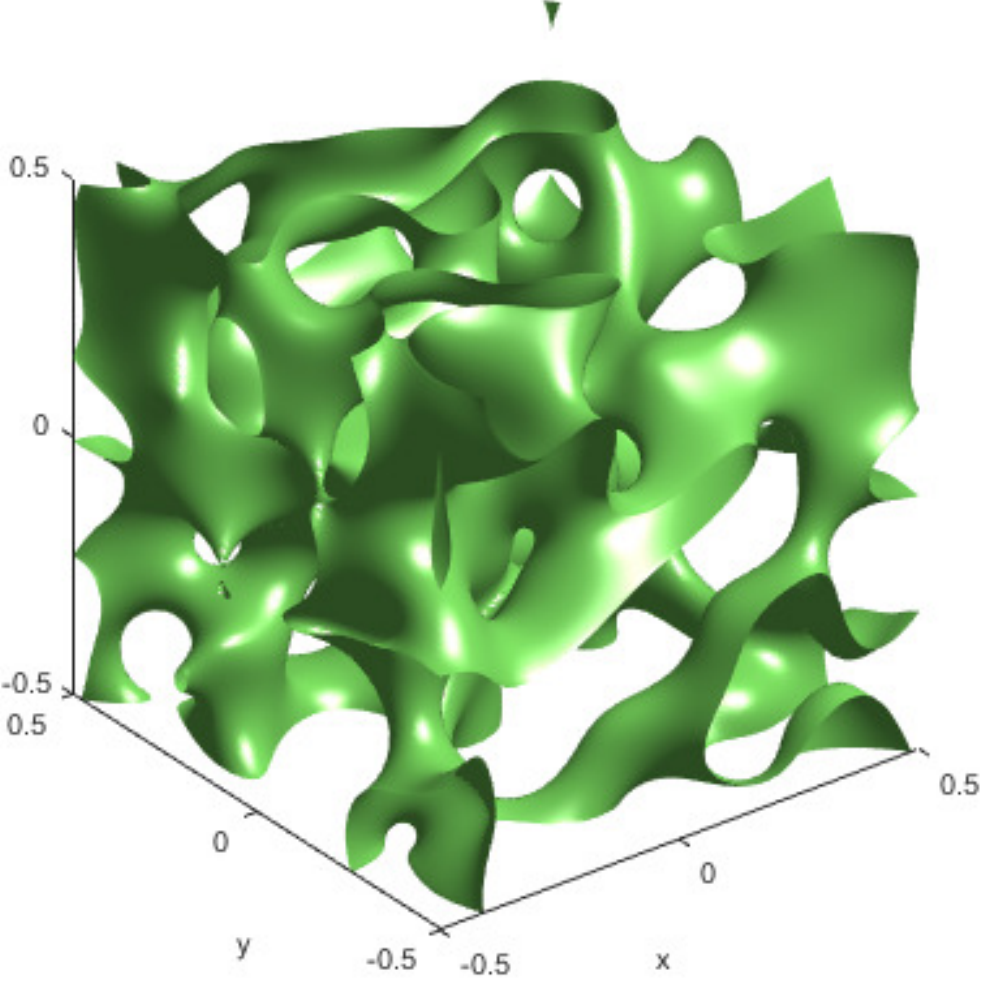}\hspace{-0.1cm}
   \includegraphics[width=0.33\textwidth]{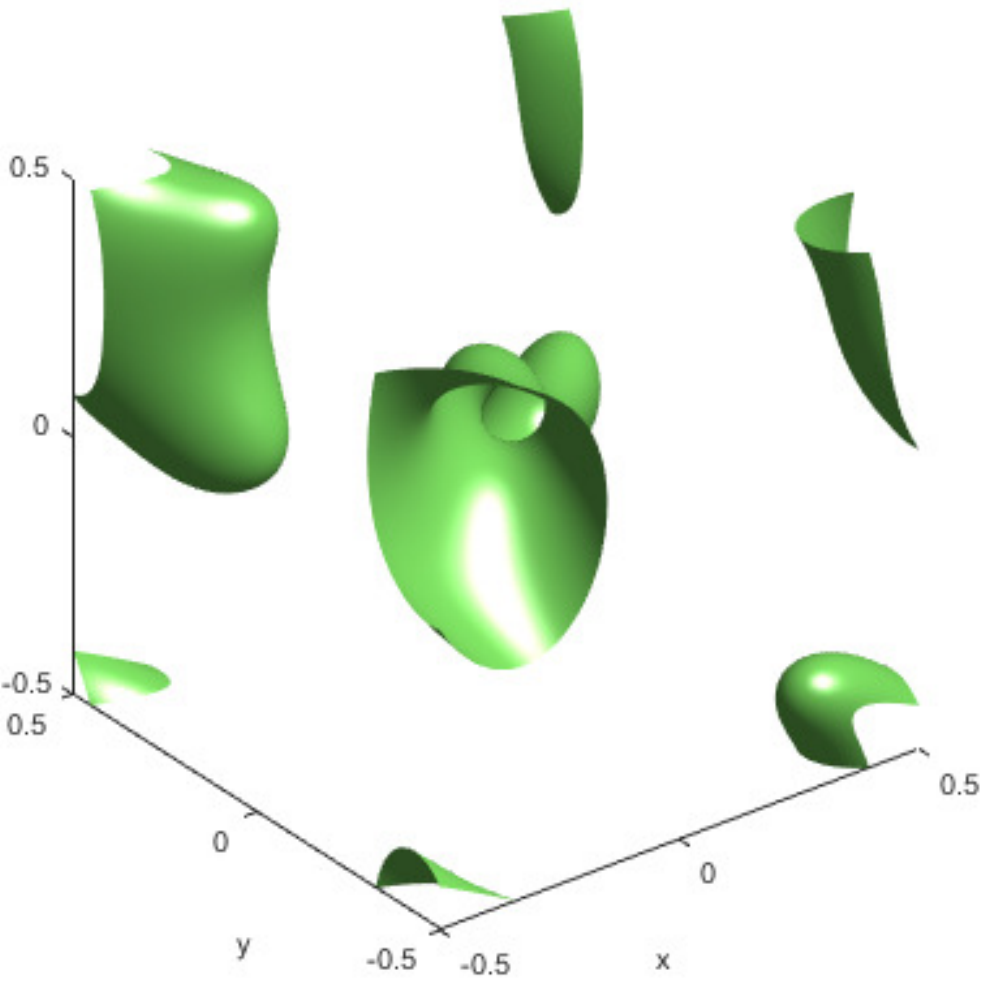}\hspace{-0.1cm}
   \includegraphics[width=0.33\textwidth]{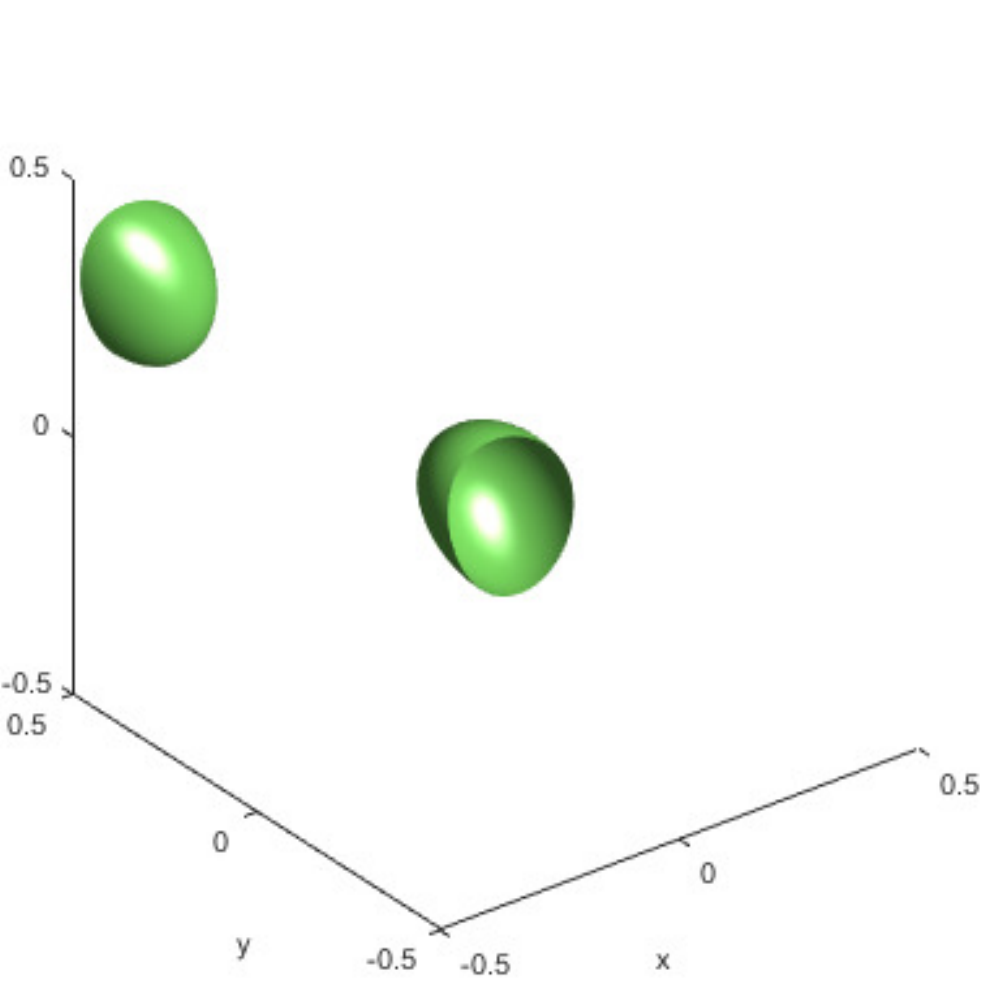}}
\caption{Evolution of the 3D phase structure obtained by using sIFRK(2,2) {with $\tau=0.01$}.
From left to right and from top to bottom: $t=1$, $5$, $10$, $50$, $240$ and $350$.}
\label{fig3d-1}
\end{figure}

\begin{figure}[!ht]
  \centering
   {\includegraphics[width=0.49\textwidth]{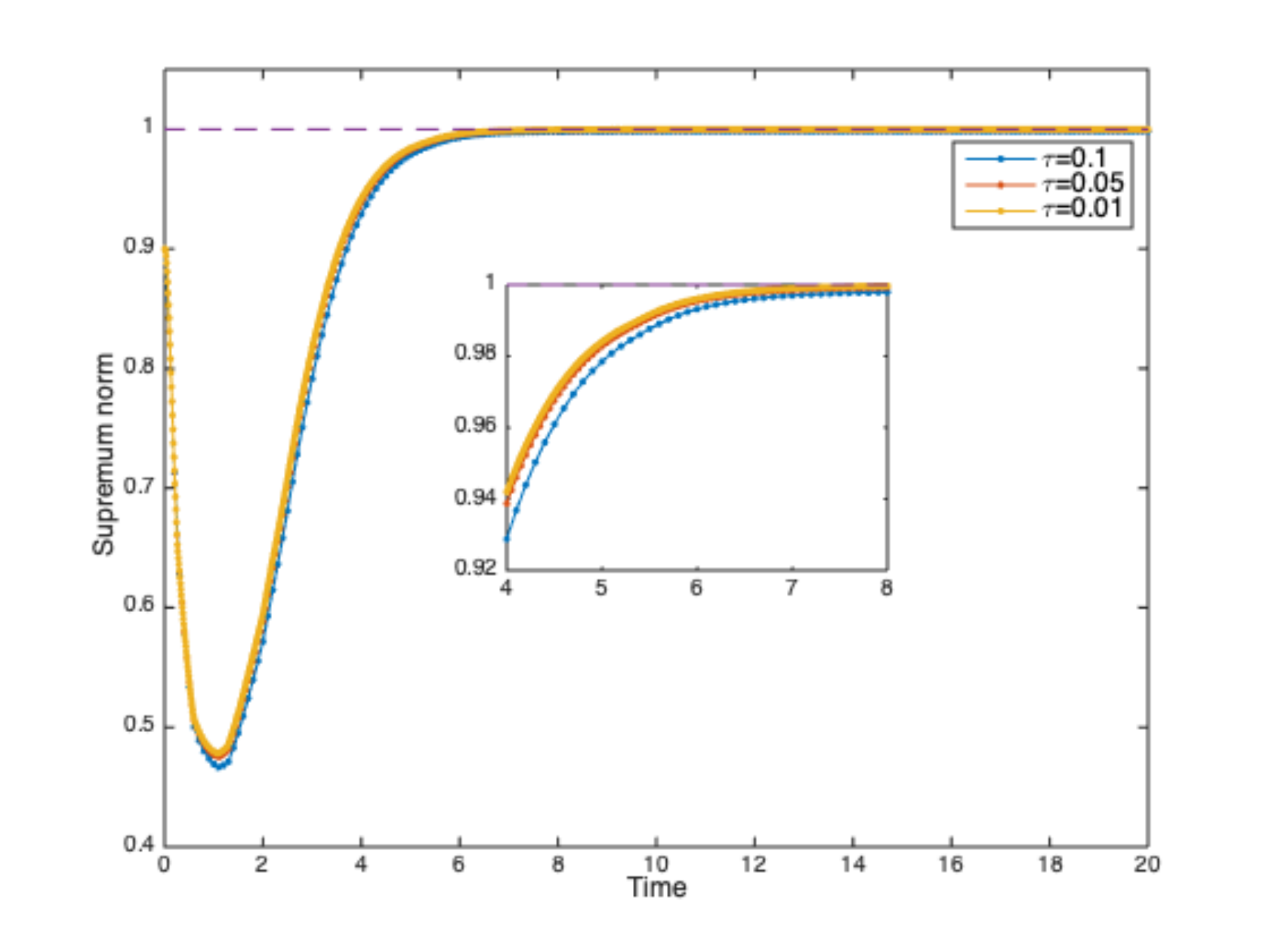}}\hspace{-0.5cm}
   {\includegraphics[width=0.49\textwidth]{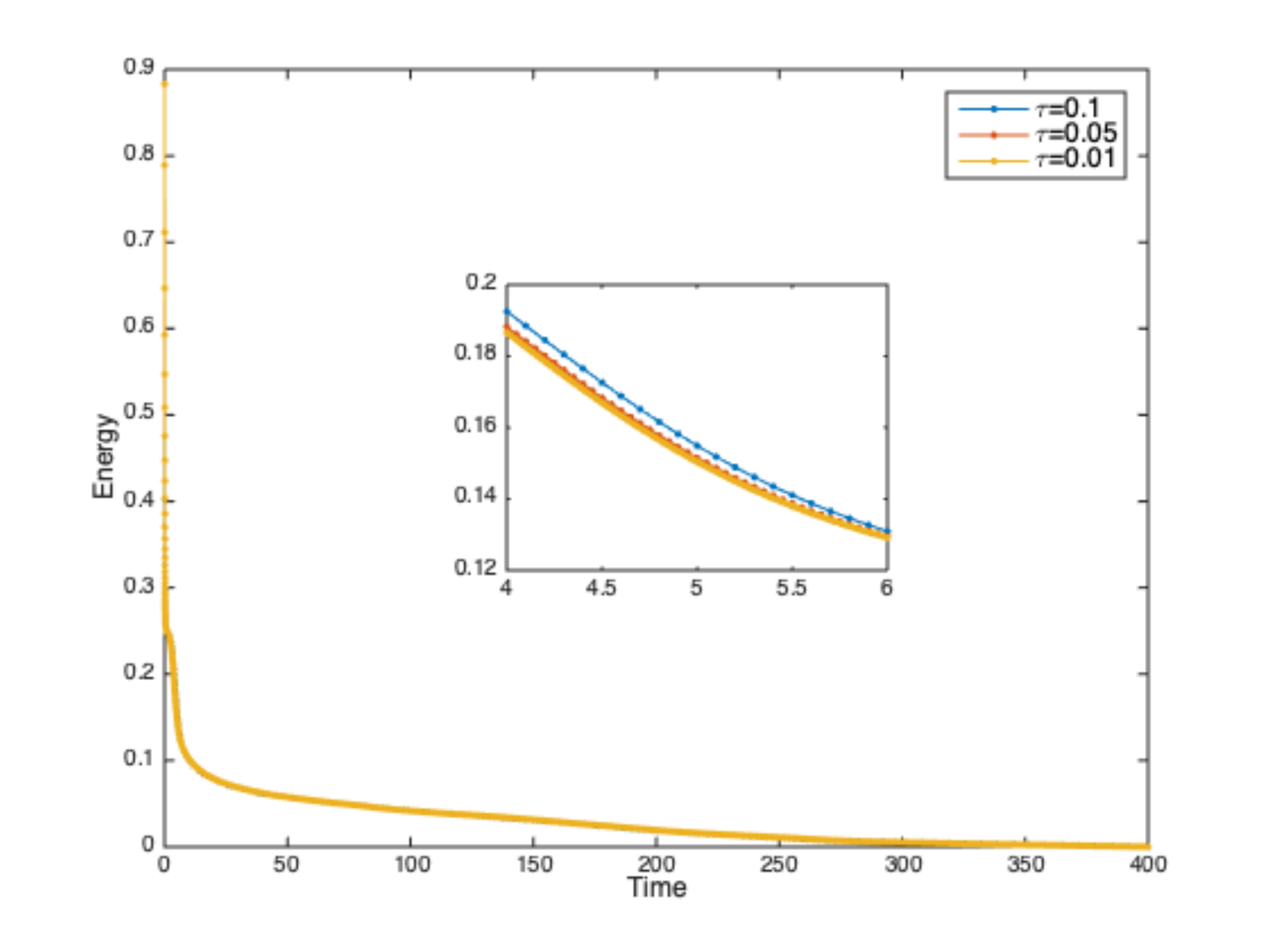}}
    \vspace{-0.2cm}
\caption{Evolutions of the supremum norm (left) and the energy (right) of the 3D simulations obtained by using sIFRK(2,2) with $\tau=0.1$, $0.05$ and $0.01$.}
\label{fig3d-2}
\end{figure}

\section{Conclusion}\label{sec6}

In this work, we first combine the linear stabilization technique with the IFRK method to develop a family of sIFRK schemes.
We derive sufficient conditions to guarantee unconditional preservation of the MBP for the sIFRK method written in different forms.
Based on these conditions, we then check various existing IFRK and SSP-IFRK schemes and identify the unconditionally MBP preserving schemes among them, as well as various numerical demonstrations. In addition,  we also
find that many existing SSP-sIFRK schemes violate these conditions except the first-order one, and thus may not be
unconditionally MBP-preserving as verified in numerical experiments.

One important question remains whether
the conditions in Theorem \ref{thm_mbp} or Theorem \ref{thm_mbp_ssp} are also necessary for unconditional MBP preservation,
which would be a natural topic for our future research. It is also worth mentioning that Assumption \ref{assump_linear} implies that the
operator $\mathcal{L}^h$ is dissipative. Discretizing $\mathcal{L}$ with the central difference method or the finite element method with  mass-lumping
satisfies Assumption  \ref{assump_linear} since the resulting discrete matrix is an $M$-matrix. As we know, the matrix corresponding to
the spectral collocation method is usually not an $M$-matrix. Thus another question is whether such assumption is  necessary
for the space-discrete system to possess the MBP and subsequently for the fully-discrete system.

\section*{Acknowledgement}

The authors would like to thank the editor and anonymous referees for their
valuable comments and suggestions, which have helped us  improve this paper a lot.

\end{document}